\documentclass[a4paper,reqno,11pt]{amsart}

\usepackage[T1]{fontenc}
\usepackage[utf8]{inputenc}
\usepackage{xspace,
	amsfonts}
\usepackage{amsmath}
\usepackage{amssymb}
\usepackage{graphicx,
	bbm,
	tikz,
	subfig}
\usepackage{xfrac}
\usepackage{dsfont}
\usepackage{amsthm}
\usepackage{mathtools}
\usepackage{enumitem}
\usepackage{verbatim}
\usepackage[autostyle]{csquotes}

\usepackage{cite}

\mathtoolsset{showonlyrefs}

\usepackage{geometry}
\newtheorem{theorem}{Theorem}[section]
\newtheorem{lemma}[theorem]{Lemma}
\newtheorem{proposition}[theorem]{Proposition}

\theoremstyle{remark}
\newtheorem{remark}[theorem]{Remark}
\theoremstyle{definition}

\newcommand{\R}{\mathbb{R}}

\newcommand{\N}{\mathbb{N}}

\newcommand{\M}{\mathcal{M}}

\newcommand{\Z}{\mathbb{Z}}

\tikzstyle{nodino}=[circle,draw,fill,inner sep=0pt,minimum size=0.5mm]
\tikzstyle{infinito}=[circle,inner sep=0pt,minimum size=0mm]
\tikzstyle{nodo}=[circle,draw,fill,inner sep=0pt, minimum size=0.5*width("k")]
\tikzstyle{nodo_vuoto}=[circle,draw,inner sep=0pt, minimum size=0.5*width("k")]
\usetikzlibrary{graphs}
\tikzset{every loop/.style={min distance=10mm,in=300,out=240,looseness=10}}
\tikzset{place/.style={circle,thick,draw=blue!75,fill=blue!20,minimum
		size=6mm}}
\tikzset{place2/.style={circle,thick,draw=red!75,fill=red!20,minimum
		size=6mm}}

\title[]{Existence of optimizers for the sharp stability constant in the logarithmic Sobolev inequality}

\author[ ]{Simone Dovetta}
\address[S. Dovetta]{Politecnico di Torino, Dipartimento di Scienze Matematiche ``G.L. Lagrange'', Corso Duca degli Abruzzi 24, 10129, Torino, Italy.}
\email{simone.dovetta@polito.it}

\author[ ]{Enrico Serra}
\address[E. Serra]{Politecnico di Torino, Dipartimento di Scienze Matematiche ``G.L. Lagrange'', Corso Duca degli Abruzzi 24, 10129, Torino, Italy.}
\email{enrico.serra@polito.it}

\begin{document}
	
	\begin{abstract}
		We study the sharp stability constant in the logarithmic Sobolev inequality, defined as the
		infimum of the logarithmic Sobolev deficit divided by the squared $L^2$-distance from the
		manifold of Gaussian optimizers. For every $N\geq 1$, first we prove that this infimum is attained
		in the Euclidean formulation of the inequality. Then, we show that every such optimizer
		decays exponentially fast at infinity, and therefore gives rise to an optimizer in the Gaussian
		formulation, where the same sharp constant appears. The proof is based on an extension of the general strategy introduced by Bianchi and Egnell for the Sobolev inequality and recently developed by K\"onig
		to prove the existence of optimizers for the corresponding sharp stability constant. The main
		difference in the logarithmic Sobolev setting is the analysis of minimizing sequences approaching the manifold of optimizers, since the entropy term $u^2\log u^2$ does not allow for a direct global second-order expansion. This is overcome by a perturbative estimate near the Gaussian manifold, which provides the compactness threshold needed to rule out the loss of compactness.
	\end{abstract}
	
	\maketitle

	\vspace{-.5cm}
	
	
	\section{Introduction}
	Let $\gamma:\R^N\to\R$ be the Gaussian $\gamma(x):=e^{-\frac\pi2|x|^2}$ and let $d\gamma=\gamma(x)^2dx$. The Gaussian logarithmic Sobolev inequality with sharp constant is
	\begin{equation}
		\label{LSI}
		\int_{\R^N}|\nabla v|^2\,d\gamma\geq\pi\int_{\R^N}v^2\log\frac{v^2}{\|v\|_{L^2(\R^N,d\gamma)}^2}\,d\gamma\qquad\forall v\in H^1(\R^N,d\gamma)\,,
	\end{equation}
	where equality is attained only by the set of optimizers 
	\[
	\M_\gamma:=\left\{a e^{b\cdot x}\,:\,a\in\R,\,b\in\R^N\right\}.
	\]
	The analogous Euclidean logarithmic Sobolev inequality can be written for instance as 
	\begin{equation}
		\label{ELSI}
		\int_{\R^N}|\nabla u|^2\,dx-N\pi\int_{\R^N}|u|^2\,dx\geq\pi\int_{\R^N}u^2\log\frac{u^2}{\|u\|_{L^2(\R^N)}^2}\,dx\qquad\forall u\in H^1(\R^N)
	\end{equation}
	with corresponding set of optimizers
	\[
	\M:=\left\{ae^{-\frac\pi2|x-b|^2}\,:\,a\in\R,\,b\in\R^N\right\}.
	\]
	Inequality \eqref{LSI} is usually traced back to Gross \cite{gross}, whereas a scale invariant version of \eqref{ELSI} appeared in \cite{weissler} but was already known in one dimension since the late Fifties due to Stam \cite{stam}. As pointed out e.g. in \cite[p. 196]{carlen} and in \cite[p. 224-225]{LL}, inequality \eqref{ELSI} follows directly by \eqref{LSI}  for functions in the form $u=\gamma v$ with $v\in H^1(\R^N,\ d\gamma)$. However, since 
	\begin{equation}
	\label{H1>H1gauss}
	\left\{u=\gamma v\,:\,v\in H^1(\R^N,d\gamma)\right\}=\left\{u\in H^1(\R^N)\,:\,\int_{\R^N}|x|^2u(x)^2\,dx<\infty\right\}\subsetneqq H^1(\R^N),
	\end{equation}
	the validity of \eqref{ELSI} for $u\in H^1(\R^N)$ with$\displaystyle \int_{\R^N}|x|^2u(x)^2\,dx=\infty$ does not come from \eqref{LSI}, but can be seen e.g. by approximation with $u_n=\eta_n u$, for suitable cut-off functions $\eta_n$.
	
	The characterization of the set of optimizers in logarithmic Sobolev inequalities was derived by Carlen in \cite{carlen}. Since then, a natural attention has been devoted to stability results for these inequalities. In this context, introducing the Gaussian deficit
	 \[
	 \delta_\gamma[v]:=\int_{\R^N}|\nabla v|^2\,d\gamma-\pi\int_{\R^N}v^2\log\frac{v^2}{\|v\|_{L^2(\R^N,d\gamma)}^2}\,d\gamma
	 \]
	for $v\in H^1(\R^N,\ d\gamma)$, and the corresponding Euclidean one
	\[
	\delta[u]:=\int_{\R^N}|\nabla u|^2\,dx-	\pi\int_{\R^N}u^2\log\frac{u^2}{\|u\|_{L^2(\R^N)}^2}\,dx-N\pi\int_{\R^N}|u|^2\,dx
	\]
	for $u\in H^1(\R^N)$, stability has been pursued either by searching for improved lower bounds for $\delta_\gamma$ and $\delta$ on suitable  subsets of $H^1(\R^N,\ d\gamma)$ and $H^1(\R^N)$, or by trying to control the deficits from below on the whole space by some distance from the manifolds of optimizers $\M_\gamma$ and $\M$. For the first approach we refer e.g. to \cite{BGRS,BDS,CCN,DT,FIL,FIPR,H} and to the comprehensive review \cite{BDS_rev}. The second approach is in the spirit of the celebrated result by Bianchi and Egnell \cite{BE} for Sobolev inequalities, where, answering to a question raised by Brezis and Lieb in \cite{BL}, it is proved that for every $N\geq3$ there exists a constant $K_N>0$ such that
	\begin{equation}
		\label{SOB}
	\int_{\R^N}|\nabla u|^2\,dx-S_N\left(\int_{\R^N}|u|^{2^*}\right)^{\frac2{2^*}}\geq K_N\inf_{B\in\mathcal{S}}\|u-B\|_{\dot{H}^1(\R^N)}^2\qquad\forall u\in H^1(\R^N)\,,
	\end{equation}
	where $\displaystyle 2^*=\frac{2N}{N-2}$ is the critical exponent for the Sobolev embedding of $H^1(\R^N)$, $S_N$ is the corresponding optimal constant, and $\mathcal S$ is the set of Aubin-Talenti functions. In the case of logarithmic Sobolev inequalities, one of the first results in this direction is an improved inequality reported by Carlen in \cite{carlen}, while various results have been obtained when the distance from the manifold of optimizers is measured by Wasserstein distances (see e.g. \cite{ELS,IM,Indrei1} and references therein). When turning to stronger norms, in the remarkable paper \cite[Corollary 1.2]{DEFFL25} (see also \cite{DEFFLrev}) it has been recently proved that 
	\begin{equation}
		\label{eq:stabLSI}
		Q:=\inf_{v\in H^1(\R^N,d\gamma)\setminus\M_\gamma}Q_\gamma(v)=\inf_{v\in H^1(\R^N,d\gamma)\setminus\M_\gamma}\frac{\delta_\gamma[v]}{d_\gamma(v,\M_\gamma)^2}>0\,, 
	\end{equation}
	where the distance from the manifold $\M_\gamma$ is taken in $L^2(\R^N,d\gamma)$
	\[
	d_\gamma(v,\M_\gamma):=\inf_{w\in\M_\gamma}\|v-w\|_{L^2(\R^N,d\gamma)}\,.
	\]
	This result is obtained in \cite{DEFFL25} as the limit of the quantitative Sobolev inequalities \eqref{SOB} in $H^1(\R^d)$ when $d\to\infty$, while an alternative proof is provided in \cite{DEFFLrev}. Moreover, under no further assumptions, this stability result is somehow optimal, in the sense that it cannot be improved replacing the distance in $L^2(\R^N,\ d\gamma)$ with stronger norms in $H^1(\R^N,\ d\gamma)$ (see e.g. \cite{Indrei2, Kim, IK}). As for \eqref{ELSI}, the stability inequality \eqref{eq:stabLSI} extends to the Euclidean setting with the same optimal constant (see \cite[Corollary 4.4]{DEFFL25})
	\begin{equation}
		\label{eq.stabELSI}
		\inf_{u\in H^1(\R^N)\setminus\M}Q(u)=	\inf_{u\in H^1(\R^N)\setminus\M}\frac{\delta[u]}{d(u,\M)^2}=Q\,,
	\end{equation}
	where
	\[
	d(u,\M):=\inf_{w\in \M}\|u-w\|_{L^2(\R^N)}\,.
	\]
	A major open problem related to these stability inequalities is to determine, or even to understand qualitatively, the optimal constants. At present, very little seems to be known in this direction, both in the Sobolev case  \eqref{SOB} and in the logarithmic one \eqref{eq:stabLSI}--\eqref{eq.stabELSI}. To the best of our knowledge, the only available result in this sense is in fact the lower bound obtained in \cite{DEFFL25}, where an explicit dimension-independent constant $\beta>0$ is determined such that the best stability constant in the Sobolev inequality \eqref{SOB} satisfies
	\[
	K_N\geq\frac\beta N \qquad\forall N\geq 3\,,
	\]
	which in turn yields the lower bound
	\[
	Q\geq\frac{\pi\beta}2
	\]
	for the best constant in \eqref{eq:stabLSI}--\eqref{eq.stabELSI}. Observe that, even though the above lower bound for $Q$ is independent of the dimension $N$ of the ambient space, as far as we know to understand whether the best constant depends on $N$ is an open question.

	A first question, which is often crucial in the study of sharp stability constants, is whether the optimal constant is attained.  For Sobolev inequalities \eqref{SOB}, this question was recently answered in the affirmative by the notable paper \cite{Konig_JEMS}. In that work the author develops a thorough analysis of the possible loss of compactness for minimizing sequences of the quotient
	\begin{equation}
	\label{eq:quoSOB}
	\displaystyle\frac{\int_{\R^N}|\nabla u|^2\,dx-S_N\left(\int_{\R^N}|u|^{2^*}\right)^{\frac2{2^*}}}{\inf_{B\in\mathcal S}\|u-B\|_{\dot{H}^1(\R^N)}^2}
	\end{equation}
	in $H^1(\R^N)\setminus\mathcal S$, showing that this can happen a priori for two different reasons. On the one hand, the sequence may become closer and closer to the manifold of Aubin-Talenti bubbles $\mathcal S$; on the other hand, the sequence can split asymptotically into the superposition of two non-interacting bubbles, one of which running away at infinity. The existence of an optimizer is then obtained by deriving two distinct compactness thresholds corresponding to these two scenarios, and by constructing an explicit competitor whose level lies strictly below both of them.
	
	\smallskip
	The main purpose of this paper is to settle the corresponding problem for the logarithmic Sobolev stability inequalities \eqref{eq:stabLSI}--\eqref{eq.stabELSI}. Precisely, our main results are the following.
	
	\begin{theorem}
	\label{thm:main}
		For every $N\geq1$, one has $Q<2\pi\log2$. Moreover, there exists $u\in H^1(\R^N)$ such that $Q(u)=Q$. 
	\end{theorem}
	\begin{theorem}
		\label{thm:main2}
		For every $N\geq1$, let $u\in H^1(\R^N)$ be as in Theorem \ref{thm:main}. Then $u$ decays exponentially fast as $|x|\to\infty$. In particular, $v:=u/\gamma\in H^1(\R^N,d\gamma)$ and $Q_\gamma(v)=Q$. 
	\end{theorem}
	Observe that, in view of \eqref{H1>H1gauss}, Theorem \ref{thm:main} does not automatically imply that $Q$ is attained in $H^1(\R^N,d\gamma)$, since a priori we do not know whether the optimizer $u$ found in $H^1(\R^N)$ satisfies 
	\[
	\displaystyle \int_{\R^N}|x|^2u(x)^2\,dx<\infty.
	\]
	This distinction is not merely technical: the Euclidean and Gaussian formulations have the same optimal constant, but compactness and decay properties are not automatically transferred from one setting to the other. In fact, we stated Theorems \ref{thm:main}--\ref{thm:main2} separately and in this order to reflect the strategy we develop to prove them. First, we tackle the problem in the Euclidean setting, establishing the existence of an optimizer for \eqref{eq.stabELSI} in $H^1(\R^N)$. Then, we exploit the Euler-Lagrange equation solved by this optimizer (see Lemma \ref{lem:EL} below) to obtain an exponential decay which is enough to prove the finiteness of  the second moment needed to recover a corresponding optimizer in the Gaussian setting \eqref{eq:stabLSI}.
	
	Our proof of Theorem \ref{thm:main} follows the same general strategy as the one developed in \cite{Konig_JEMS} for the Sobolev case, but its implementation in the logarithmic Sobolev setting requires a substantially different analysis at the perturbative level. The main new difficulty is the analysis of minimizing sequences approaching the manifold $\M$ of optimizers. In contrast with the Sobolev case, the entropy term does not allow for a direct global second-order expansion around $\M$. In fact, it was already pointed out in \cite{DEFFL25} that ``the Bianchi--Egnell strategy has so far not been applied to the logarithmic Sobolev inequality, probably because $u\mapsto u^2 \log u^2$ is not twice differentiable at the origin''. The key point of our argument is Proposition \ref{prop:BE}, where this difficulty is overcome by a localized analysis of perturbations of the Gaussian profile, separating the region in which the perturbation is genuinely small from the region in which the singular behavior of the entropy becomes relevant.
	
	\smallskip
	The remainder of the paper is organized as follows. In Section \ref{sec:Q<2log2} we prove the strict upper bound
	$Q<2\pi\log 2$, where $2\pi\log 2$ is the threshold associated with two asymptotically independent Gaussian bubbles. Section \ref{sec:estM} contains the perturbative Bianchi--Egnell type analysis close to the manifold $\M$. In Section \ref{sec:comp} we combine the analysis of the previous sections with a
	concentration-compactness argument to complete the proof of Theorem \ref{thm:main}. Finally, in Section \ref{sec:main2} we derive the Euler--Lagrange equation for an optimizer, prove its exponential decay, and show that the sharp constant in the stability inequality is attained in the Gaussian space $H^1(\R^N,\ d\gamma)$ too, as stated in Theorem \ref{thm:main2}.
	
	\bigskip
	{\bf Notation.} Throughout, we will use shorthand notation like $\|u\|_p$ to denote $L^p$ norms in $\R^N$, omitting the explicit reference to the domain of integration whenever possible. Moreover, the complement of the set $A$ will be denoted by $A^c$, and $\log_\pm$ will denote respectively the positive and negative part of the logarithmic function.
	
	\section{An upper bound on $Q$}
	\label{sec:Q<2log2}
	
	This section is devoted to the proof of the first part of Theorem \ref{thm:main}, as stated in the next proposition.
	\begin{proposition}
		\label{prop:Q<2pilog2}
		For every $N\geq1$, one has $Q<2\pi\log2$.
	\end{proposition}
	The threshold $2\pi\log2$ corresponds to the energy level of two asymptotically disjoint Gaussian functions. To show that the sharp constant $Q$ in \eqref{eq.stabELSI} is strictly smaller than this value, we consider the family of functions
	\begin{equation}
	\label{eq:ua}
	u_a(x):=\gamma(x)+\gamma(x-ae_1)=e^{-\frac\pi2|x|^2}+e^{-\frac\pi2|x-ae_1|^2}
	\end{equation}
	with $a>0$ and $e_1=(1,0,\dots,0)\in\R^N$, and we establish an asymptotic expansion for $Q(u_a)$ in the limit $a\to+\infty$.
	
	\begin{remark}
		\label{rem:d_assunta}
		By definition, for every $u\in H^1(\R^N)$
		\[
		d(u,\M)^2=\inf_{a\in\R,\,b\in\R^N}\int_{\R^N}\left|u(x)-ae^{-\frac\pi2|x-b|^2}\right|^2\,dx=\|u\|_2^2-m(u)
		\]
		where
		\[
		m(u):=\sup_{b\in\R^N}\left(\int_{\R^N}u(x)e^{-\frac\pi2|x-b|^2}\,dx\right)^2.
		\]
		It then follows directly that
		\[
		d(u,\M)<\|u\|_2^2\qquad\forall u\not\equiv0\,.
		\]
		Indeed, if $d(u,\M)=\|u\|_2^2$, then $m(u)=0$, i.e.
		\[
		\int_{\R^N}u(x)e^{-\frac\pi2|x-b|^2}\,dx=0\qquad\forall b\in\R^N\,.
		\]
		In terms of the Fourier transform, this means that $\widehat{u*\gamma}\equiv0$, but since $\widehat{u*\gamma}=\widehat{u}\,\widehat{\gamma}$ and $\widehat{\gamma}=\gamma>0$ on $\R^N$, it then follows $\widehat{u}\equiv0$, that is $u\equiv0$. Hence, $m(u)>0$ for every $u\in H^1(\R^N)\setminus\left\{0\right\}$, and it is easily seen that it is attained, i.e. there exists $b\in \R^N$ (depending on $u$) such that $\displaystyle m(u)=\left(\int_{\R^N}u(x)e^{-\frac\pi2|x-b|^2}\,dx\right)^2$. 
	\end{remark}
	
	To prove Proposition \ref{prop:Q<2pilog2} note first that it is enough to deal with the case $N=1$. Indeed, if $N\geq2$, since
	\[
	u_a(x)=e^{-\frac\pi2\sum_{i\neq1}x_i^2}\left(e^{-\frac\pi2 x_1^2}+e^{-\frac\pi2(x_1-a)^2}\right)=:e^{-\frac\pi2|x'|^2}v_a(x_1)\,,
	\]
	writing $\displaystyle x=(x_1,x')\in\R^N$ with $x_1\in\R,\, x'\in\R^{N-1}$, we have
	\begin{equation}
		\label{eq:massuaN>1}
		\begin{split}
		\int_{\R^N}|u_a|^2\,dx=\int_{\R^N}e^{-\pi|x'|^2}v_a(x_1)^2dx=\int_{\R} v_a(x_1)^2\,dx_1\cdot \Pi_{i\neq1}\int_{\R}e^{-\pi x_i^2}dx_i=\int_\R v_a(x_1)^2dx_1
		\end{split}
	\end{equation}
	and
	\begin{equation}
		\label{eq:entuaN>1}
		\begin{split}
			\int_{\R^N}u_a^2\log u_a^2\,dx&\,=\int_{\R^N}e^{-\pi|x'|^2}v_a(x_1)^2\log e^{-\pi|x'|^2}v_a(x_1)^2dx\\
			&\,=-\pi\int_{\R^N}e^{-\pi|x'|^2}v_a(x_1)^2|x'|^2\,dx+\int_{\R^N}e^{-\pi|x'|^2}v_a(x_1)^2\log v_a(x_1)^2dx\\
			&\,=-\pi\int_{\R}v_a(x_1)^2dx_1\sum_{i\neq1}\int_{\R}x_i^2e^{-\pi x_i^2}\,dx_i+\int_{\R}v_a(x_1)^2\log v_a(x_1)^2dx_1\\
			&\,=-\frac{N-1}{2}\int_{\R}v_a(x_1)^2dx_1+\int_{\R}v_a(x_1)^2\log v_a(x_1)^2dx_1\,,
		\end{split}
	\end{equation}
	where we repeatedly used the identities
	\begin{equation}
		\label{eq:intexp}
		\int_{\R}e^{-\pi x^2}\,dx=1\qquad\text{and}\qquad\int_{\R}x^2e^{-\pi x^2}\,dx=\frac1{2\pi}\,.
	\end{equation}
	Moreover, since
	\[
		\frac{\partial u_a}{\partial x_1}(x)=e^{-\frac\pi2|x'|^2}v_a'(x_1)\,,\qquad\frac{\partial u_a}{\partial x_j}(x)=-\pi x_j e^{-\frac\pi2|x'|^2}v_a(x_1)\,,\quad j=2,\dots,N,
	\]
	we obtain
	\begin{equation}
		\label{eq:graduaN>1}
		\begin{split}
			\int_{\R^N}|\nabla u_a|^2\,dx&\,=\int_{\R^N}e^{-\pi|x'|^2}\left|v_a'(x_1)\right|^2\,dx+\int_{\R^N}v_a(x_1)^2\pi^2e^{-\pi|x'|^2}\sum_{i\neq1}x_i^2\,dx\\
			&\,=\int_{\R}\left|v_a'(x_1)\right|^2dx_1+\pi^2\int_{\R}v_a(x_1)^2dx_1\sum_{i\neq1}\int_{\R}x_i^2 e^{-\pi x_i^2}\,dx_i\\
			&\,=\int_{\R}\left|v_a'(x_1)\right|^2dx_1+\frac{(N-1)\pi}2\int_{\R}v_a(x_1)^2dx_1\,.
		\end{split}
	\end{equation}
	Finally, since by Remark \ref{rem:d_assunta}
	\[
	d(u_a,\M)^2=\|u_a\|_2^2-m(u_a)=\|u_a\|_2^2-\left(\sup_{b\in R^N}\int_{\R^N}e^{-\frac\pi2|x'|^2}v_a(x_1)e^{-\frac\pi2|x-b|^2}\,dx\right)^2
	\] 
	and
	\[
	\begin{split}
	\int_{\R^N}e^{-\frac\pi2|x'|^2}v_a(x_1)e^{-\frac\pi2|x-b|^2}\,dx&\,=\int_{\R}v_a(x_1)e^{-\frac\pi2(x_1-b_1)^2}\,dx_1\cdot\Pi_{i\neq1}\int_{\R}e^{-\frac\pi2\left(x_i^2+(x_i-b_i)^2\right)}dx_i\\
	&\,=\int_{\R}v_a(x_1)e^{-\frac\pi2(x_1-b_1)^2}\,dx_1\cdot\Pi_{i\neq1}e^{-\frac\pi4 b_i^2}\,,
	\end{split}
	\]
	we have $m(u_a)=m(v_a)$, that together with \eqref{eq:massuaN>1} shows that
	\[
	d(u_a,\M)=d(v_a,\M)
	\]
	(where with a slight abuse of notation we denoted with $\M$ the manifold of gaussian optimizers both in $\R^N$ and in $\R$).
	Combining with \eqref{eq:massuaN>1}, \eqref{eq:entuaN>1} and \eqref{eq:graduaN>1} yields
	\[
	\begin{split}
	Q(u_a)&\,=\frac{\delta[u_a]}{d(u_a,\M)^2}\\
	&\,=\frac{\|v_a'\|_2^2+\frac{(N-1)\pi}2\|v_a\|_2^2+\frac{(N-1)\pi}2\|v_a\|_2^2-\pi\int_{\R}v_a^2\log v_a^2\,dx_1+\pi\|v_a\|_2^2\log\|v_a\|_2^2-N\pi\|v_a\|_2^2}{d(v_a,\M)^2}\\
	&\,=\frac{\|v_a'\|_2^2-\pi\int_{\R}v_a^2\log v_a^2\,dx_1+\pi\|v_a\|_2^2\log\|v_a\|_2^2-\pi\|v_a\|_2^2}{d(v_a,\M)^2}=\frac{\delta[v_a]}{d(v_a,\M)^2}=Q(v_a)\,.
	\end{split}
	\]
	Hence, if for any $a>0$ we have $Q(v_a)<2\pi\log2$, Proposition \ref{prop:Q<2pilog2} is proved in every dimension.
	
	\begin{proof}[Proof of Proposition \ref{prop:Q<2pilog2} with $N=1$]
		Given $a>0$, let $u_a\in H^1(\R)$ be as in \eqref{eq:ua}. Setting $\displaystyle q:=q(a)=e^{-\frac\pi4a^2}$, a direct computation immediately shows that
		\begin{equation}
			\label{eq:massua}
			\int_\R u_a^2\,dx= 2+2\int_\R e^{-\frac\pi2(x^2+(x-a)^2)}\,dx=2+2q
		\end{equation}
		and (recalling \eqref{eq:intexp})
		\begin{equation}
			\label{eq:derua}
			\begin{split}
			\int_{\R}|u_a'|^2\,dx&\,=\pi^2\int_{\R}x^2e^{-\pi x^2}\,dx+\pi^2\int_{\R}(x-a)^2e^{-\pi(x-a)^2}\,dx+2\pi^2\int_{\R}x(x-a)e^{-\frac\pi2\left(x^2+(x-a)^2\right)}\,dx\\
			&\,=\pi+2\pi^2\int_{\R}x(x-a)e^{-\frac\pi2\left(x^2+(x-a)^2\right)}\,dx=\pi+2\pi^2q\int_\R\left(y+\frac a2\right)\left(y-\frac a2\right)e^{-\pi y^2}\,dy\\
			&\,=\pi+2\pi^2q\left(\frac1{2\pi}-\frac{a^2}4\right)=\pi+\pi q-\frac{\pi^2a^2q}2\,.
			\end{split}
		\end{equation}
		We need to derive analogous expansions for the distance between $u_a$ and the manifold $\M$ and for the entropy term. For the sake of clarity, the rest of the proof is divided in some steps.
		
		\smallskip
		{\em Step 1: expansion of $d(u_a,\M)$.} By Remark \ref{rem:d_assunta}, for every $a>0$ there exists $b_a\in\R$ such that
		\[
		m(u_a)=\sup_{b\in\R}\left(\int_{\R}u_a(x)e^{-\frac\pi2(x-b)^2}dx\right)^2=\left(\int_{\R}u_a(x)e^{-\frac\pi2(x-b_a)^2}dx\right)^2=\left(e^{-\frac\pi4b_a^2}+e^{-\frac\pi4(a-b_a)^2}\right)^2,
		\]
		so that by \eqref{eq:massua}
		\begin{equation}
			\label{eq:distua_int}
		d(u_a,\M)^2=2+2q-\left(e^{-\frac\pi4b_a^2}+e^{-\frac\pi4(a-b_a)^2}\right)^2.
		\end{equation}
		Observe that, for every $a>0$, the function $f_a(s):=e^{-\frac\pi4s^2}+e^{-\frac\pi4(a-s)^2}$ is symmetric about $s=a/2$. Hence, with no loss of generality we have that a global maximum point $b_a$ of $f_a$ is such that $b_a\in[0,a/2]$. Moreover, it is readily seen that $|b_a|=o(1/a)$ as $a\to+\infty$, since if it where $|b_a|\geq C/a$ for some $C>0$, then as $a\to+\infty$ we would have 
		\[
		f_a(b_a)\leq e^{-\frac{\pi C^2}{2a^2}}+e^{-\frac\pi{16}a^2}+o(e^{-\frac\pi{16}a^2})=1-\frac{\pi C^2}{2a^2}+o\left(\frac1{a^2}\right)<1\,,
		\]
		which is impossible since $f_a(b_a)\geq f_a(0)>1$ for every $a$. Therefore, since $b_a$ is a critical point for $f_a$, combining $b_a=o(1/a)$ as $a\to+\infty$ with
		\[
		f_a'(b_a)=-\frac\pi2e^{-\frac\pi4 b_a^2}\left(b_a-(a-b_a)e^{-\frac\pi4(a^2-2ab_a)}\right)=0
		\]
		yields
		\[
		b_a=ae^{-\frac\pi4 a^2}+o\left(ae^{-\frac\pi4 a^2}\right)=aq+o(aq)\qquad\text{as }a\to+\infty\,,
		\]
		in turn entailing
		\[
		\begin{split}
		f_a(b_a)^2&\,=\left(e^{-\frac\pi4a^2q^2+o\left(a^2q^2\right)}+e^{-\frac\pi4\left(a-aq+o\left(aq\right)\right)^2}\right)^2\\
		&\,=\left(1-\frac\pi4a^2q^2+q+\frac\pi2a^2q^2+o\left(a^2q^2\right)\right)^2=\left(1+q+\frac\pi4a^2q^2+o\left(a^2q^2\right)\right)^2\\
		&\,=1+2q+\frac\pi2a^2q^2+o\left(a^2q^2\right)=1+2q+\frac\pi2a^2q^2+o\left(a^2q^2\right)
		\end{split}
		\]
		as $a\to+\infty$. Plugging into \eqref{eq:distua_int} and recalling the definition of $q$ we obtain
		\begin{equation}
			\label{eq:distua}
			d(u_a,\M)^2=1-\frac\pi2a^2e^{-\frac\pi2a^2}+o\left(a^2e^{-\frac\pi2a^2}\right)=1-o(q)\qquad\text{as }a\to+\infty.
		\end{equation}
	
		{\em Step 2: expansion of $\displaystyle\int_{\R}u_a^2\log u_a^2\,dx$.} Changing the variable $\displaystyle x\mapsto x+\frac a2$ we rewrite
		\[
		\int_{\R}u_a(x)^2\log u_a(x)^2\,dx=\int_\R \left(e^{-\frac\pi2\left(x-\frac a2\right)^2}+e^{-\frac\pi2\left(x+\frac a2\right)^2}\right)^2\log\left(e^{-\frac\pi2\left(x-\frac a2\right)^2}+e^{-\frac\pi2\left(x+\frac a2\right)^2}\right)^2dx
		\] 
		and we note that, for every $x\in\R$,
		\[
		e^{-\frac\pi2\left(x-\frac a2\right)^2}+e^{-\frac\pi2\left(x+\frac a2\right)^2}=e^{-\frac\pi2x^2}e^{-\frac\pi8a^2}\left(e^{\frac\pi2ax}+e^{-\frac\pi2ax}\right)=2e^{-\frac\pi8a^2}e^{-\frac\pi2x^2}\cosh(\pi ax/2).
		\]
		Plugging into the above formula for the entropy term gives
		\begin{equation}
			\label{eq:splitent}
			\int_{\R}u_a^2\log u_a^2\,dx=\int_\R 4qe^{-\pi x^2}\cosh^2(\pi ax/2)\log\left(4qe^{-\pi x^2}\cosh^2(\pi ax/2)\right)dx=:A+B+C
		\end{equation}
		with
		\[
		\begin{split}
			A:=&\,4q\log(4q)\int_\R e^{-\pi x^2}\cosh^2(\pi ax/2)\,dx\\
			B:=&-4\pi q\int_\R x^2e^{-\pi x^2}\cosh^2(\pi ax/2)\,dx\\
			C:=&\,4q\int_\R e^{-\pi x^2}\cosh^2(\pi ax/2)\log\cosh^2(\pi ax/2)dx\,.
		\end{split}
		\]
		Now, recalling \eqref{eq:intexp},
		\begin{equation}
			\label{eq:A}
		\begin{split}
		A=&\,q\log(4q)\int_\R e^{-\pi x^2}(e^{\pi ax/2}+e^{-\pi ax/2})^2\,dx=q\log(4q)\int_\R e^{-\pi x^2}(e^{\pi ax}+e^{-\pi ax}+2)\,dx\\
		=&\,2q\log(4q)\int_\R e^{-\pi x^2}\,dx+q\log(4q)\left(\int_\R e^{-\pi(x^2-ax)}\,dx+\int_\R e^{-\pi(x^2+ax)}\,dx\right)\\
		=&\,2q\log(4q)+\log(4q)\left(\int_\R e^{-\pi\left(x-\frac a2\right)^2}\,dx+\int_\R e^{-\pi\left(x+\frac a2\right)^2}\right)=(2+2q)\log(4q)\,.
		\end{split}
		\end{equation}
		Furthermore, 
		\begin{equation}
			\label{eq:B}
		\begin{split}
			B=&\,-\pi q\int_\R x^2e^{-\pi x^2}(e^{\pi ax/2}+e^{-\pi ax/2})^2\,dx=-\pi q\int_\R x^2e^{-\pi x^2}(e^{\pi ax}+e^{-\pi ax}+2)\,dx\\
			=&\,-2\pi q\int_\R x^2e^{-\pi x^2}\,dx -\pi\int_\R x^2\left(e^{-\pi\left(x-\frac a2\right)^2}+e^{-\pi\left(x+\frac a2\right)^2}\right)dx\\
			=&\,-q-\pi\int_\R \left(y+\frac a2\right)^2e^{-\pi y^2}\,dy-\pi\int_\R \left(y-\frac a2\right)^2e^{-\pi y^2}\,dy\\
			=&\,-q-2\pi\left(\int_\R y^2 e^{-\pi y^2}\,dy+\frac{a^2}4\int_\R e^{-\pi y^2}\,dy\right)=-q-2\pi\left(\frac1{2\pi}+\frac{a^2}4\right)=-q-1-\frac{\pi a^2}{2}\,,
		\end{split}
		\end{equation}
		where we used again \eqref{eq:intexp} and the fact that $\displaystyle \int_\R y e^{-\pi y^2}\,dy=0$.
		
		To deal with the last term in \eqref{eq:splitent}, observe that
		\[
		\cosh z =\frac{e^{|z|}}2\left(1+e^{-2|z|}\right)\qquad\forall z\in\R\,,
		\]
		so that
		\[
		\log\cosh^2z=2\log\cosh z=2|z|-\log4+2\log\left(1+e^{-2|z|}\right)\,.
		\]
		Hence, 
		\begin{equation}
			\label{eq:splitC}
		\begin{split}
		C=&\,4q\int_{\R}e^{-\pi x^2}\cosh^2(\pi ax/2)\pi a|x|\,dx-4q\log4\int_\R e^{-\pi x^2}\cosh^2(\pi ax/2)\,dx\\
		&\qquad\qquad+8q\int_\R e^{-\pi x^2}\cosh^2(\pi ax/2)\log(1+e^{-\pi a |x|})\,dx\\
		&\,\geq 4q\int_{\R}e^{-\pi x^2}\cosh^2(\pi ax/2)\pi a|x|\,dx-4q\log4\int_\R e^{-\pi x^2}\cosh^2(\pi ax/2)\,dx\\
		&\,=4q\int_{\R}e^{-\pi x^2}\cosh^2(\pi ax/2)\pi a|x|\,dx-(2+2q)\log4\,,
		\end{split}
		\end{equation}
		where the inequality is justified by the positivity of the third integral in the computation of $C$, and the last equality is obtained repeating the same computations developed for $A$ above. Moreover, 
		\[
		\begin{split}
			4q\int_{\R}e^{-\pi x^2}\cosh^2(\pi ax/2)\pi a|x|\,dx&\,=\pi aq\int_\R|x|e^{-\pi x^2}(e^{\pi ax}+e^{-\pi ax}+2)\,dx\\
			&\,=2\pi a q\int_\R |x|e^{-\pi x^2}+2\pi aq \int_0^{+\infty} xe^{-\pi x^2}\left(e^{\pi a x}+e^{-\pi ax}\right)dx\\
			&\,=2aq+2\pi a\int_0^{+\infty}xe^{-\pi\left(x-\frac a2\right)^2}\,dx+2\pi a\int_0^{+\infty}xe^{-\pi\left(x+\frac a2\right)^2}\,dx\,,
		\end{split}
		\]
		and since for every fixed $c\in\R$
		\[
		2\pi\int_0^{+\infty}xe^{-\pi(x+c)^2}\,dx=e^{-\pi c^2}-2\pi c\int_c^{+\infty}e^{-\pi y^2}\,dy
		\]
		we obtain
		\[
		\begin{split}
			4q\int_{\R}e^{-\pi x^2}\cosh^2(\pi ax/2)\pi a|x|\,dx&\,=4aq+\pi a^2\int_{-\frac a2}^{+\infty}e^{-\pi y^2}\,dy-\pi a^2\int_{\frac a2}^{+\infty}e^{-\pi y^2}\,dy\\
			&\,=4aq+\pi a^2\int_{-\frac a2}^{\frac a2}e^{-\pi y^2}\,dy=4aq+\pi a^2\left(1-\text{\normalfont erfc}(\sqrt{\pi}a/2)\right)\,,
		\end{split}
		\]
		where by definition $\displaystyle \text{\normalfont erfc}(z)=\frac2{\sqrt\pi}\int_z^{+\infty}e^{-t^2}\,dt$. Relying on the well-known expansion 
		\[
		\text{\normalfont erfc}(z)=\frac{e^{-z^2}}{z\sqrt\pi}\left(1+O\left(\frac1{z^2}\right)\right)\qquad\text{as }z\to+\infty
		\]
		we have
		\[
		4q\int_{\R}e^{-\pi x^2}\cosh^2(\pi ax/2)\pi a|x|\,dx=4aq+\pi a^2\left(1-\frac{2q}{\pi a}+O\left(\frac q{a^3}\right)\right)
		\]
		as $a\to+\infty$. Plugging into \eqref{eq:splitC} yields for $a$ large enough
		\[
		C\geq -(2+2q)\log4+4aq+\pi a^2\left(1-\frac{2q}{\pi a}+O\left(\frac q{a^3}\right)\right)\,,
		\]
		and combining with \eqref{eq:splitent}, \eqref{eq:A} and \eqref{eq:B}
		\begin{equation}
			\label{eq:entua}
		\begin{split}
			\int_\R u_a^2\log u_a^2\,dx&\,\geq (2+2q)\log q-q-1+2aq+\frac{\pi a^2}2+O\left(\frac qa\right)\\
			&\,=-(1+q)\frac{\pi a^2}2-q-1+2aq+\frac{\pi a^2}2+O\left(\frac qa\right)\\
			&\,=-\frac{\pi a^2q}2-q-1+2aq+O\left(\frac qa\right)\qquad\text{as }a\to+\infty.
		\end{split}
		\end{equation}
	
		{\em Step 3: conclusion.} Coupling \eqref{eq:massua}, \eqref{eq:derua} and \eqref{eq:entua} we have
		\[
		\begin{split}
			\delta[u_a]\leq&\,\pi+\pi q-\frac{\pi^2 a^2q}{2}+\frac{\pi^2a^2q}2+\pi q+\pi-2\pi aq+2\pi(1+q)\log(2(1+q))-2\pi(1+q)+O\left(\frac qa\right)\\
			=&\,-2\pi aq+2\pi(1+q)\log2+2\pi(1+q)\log(1+q)+O\left(\frac qa\right)=2\pi\log2-2\pi aq+O(q)
		\end{split}	 
		\]
		for sufficiently large $a$, and together with \eqref{eq:distua} this implies
		\[
		\begin{split}
		Q(u_a)\leq \frac{2\pi\log2-2\pi aq+O(q)}{1-o(q)}&\,=\left(2\pi\log2-2\pi aq+O(q)\right)(1+o(q))\\
		&\,=2\pi\log2-2\pi aq+o(aq)<2\pi\log2
		\end{split}
		\]
		as soon as $a$ is large enough, concluding the proof of Proposition \ref{prop:Q<2pilog2}.
	\end{proof}
	\section{A priori estimates close to $\M$}
	\label{sec:estM}
	The aim of this section is to obtain a general lower bound on $Q(u)$ for functions $u$ close to the manifold $\M$ of log-Sobolev optimizers in $H^1(\R^N)$.  Precisely, the main result of this section is the following.
	\begin{proposition}
		\label{prop:BE}
		Let $(u_n)_n\subset H^1(\R^N)$ be such that 
		\begin{equation}
		\label{eq:hyp_BE}
		\inf_n\|u_n\|_2>0,\qquad\sup_n Q(u_n)<\infty, \qquad\text{and}\qquad \lim_n d(u_n,\M)=0\,.
		\end{equation}
		Then
		\begin{equation}
			\label{eq:Q>2}
		\liminf_n Q(u_n)\geq2\pi\,.
		\end{equation}
	\end{proposition}
	The proof of Proposition \ref{prop:BE} relies on the next preliminary lemma.
	\begin{lemma}
		\label{lem:boundBE}
	For every $n$, let $u_n:=\gamma+\varepsilon_n\phi_n$, where $\varepsilon_n\to0$ as $n\to\infty$ and $\phi_n\in H^1(\R^N)$ is such that
	\begin{equation}
	\label{eq:phiL2=1}
	\|\phi_n\|_2=1
	\end{equation}
	and
	\begin{equation}
	\label{eq:phi_ortT}
	\int_{\R^N}\phi_n\gamma\,dx=\int_{\R^N}\phi_n\partial_i\gamma\,dx=0\qquad\forall i=1,\dots,N\,.
	\end{equation}
	If $\displaystyle\sup_n Q(u_n)<\infty$, then $(\phi_n)_n$ is bounded in $H^1(\R^N)$ and, up to subsequences, $\phi_n\to\phi$ in $L^2(\R^N)$ as $n\to\infty$, for some $\phi\in H^1(\R^N)$ such that
	\begin{equation}
	\label{eq:lim_ortM}
	\int_{\R^N}\phi\gamma\,dx=\int_{\R^N}\phi\partial_i\gamma\,dx=0\qquad\forall i=1,\dots,N\,.
	\end{equation}
 	\end{lemma}
 	\begin{proof}
 	Observe that, by  the definition of $u_n$, $\varepsilon_n\to0$, \eqref{eq:phiL2=1} and  \eqref{eq:phi_ortT}, 
 	\begin{equation}
 	\label{eq:splitun1}
 	\begin{split}
 		\int_{\R^N}|\nabla u_n|^2\,dx&\,=\int_{\R^N}|\nabla \gamma|^2\,dx+2\varepsilon_n\int_{\R^N}\nabla\gamma\cdot\nabla \phi_n+\varepsilon_n^2\int_{\R^N}|\nabla\phi_n|^2\,dx\\
 		\int_{\R^N}|u_n|^2\,dx&\,=\int_{\R^N}\gamma^2\,dx+\varepsilon_n\int_{\R^N}\phi_n^2\,dx=1+\varepsilon_n^2\\
 		\|u_n\|_2^2\log\|u_n\|_2^2&\,=(1+\varepsilon_n^2)\log(1+\varepsilon_n^2)=\varepsilon_n^2+o(\varepsilon_n^2)\,,
 	\end{split}
 	\end{equation}
 	and, since \eqref{eq:phi_ortT} means that $\phi_n$ is orthogonal to the tangent space $T_\gamma\M$ to $\M$ at $\gamma$, it is also easy to see that
 	\[
 	d(u_n,\M)^2=d(u_n,\gamma)^2+o(\varepsilon_n^2)=\varepsilon_n^2\|\phi_n\|_2^2+o(\varepsilon_n^2)=\varepsilon_n^2+o(\varepsilon_n^2)\,.
 	\]
 	Now, setting $h(s):=s^2\log s^2$ and recalling that, by the optimality of $\gamma$ in \eqref{ELSI}, it holds
 	\[
 	\int_{\R^N}|\nabla\gamma|^2\,dx-N\pi\int_{\R^N}\gamma^2\,dx=\pi\int_{\R^N}h(\gamma)\,dx=\pi\int_{\R^N}\gamma^2\log\gamma^2\,dx
 	\]
 	and
 	\[
 	\begin{split}
 	2\int_{\R^N}\nabla\gamma\cdot\nabla w&\,=\pi\int_{\R^N}h'(\gamma)w\,dx\\
 	&\,=2\pi\int_{\R^N}\gamma(\log\gamma^2+1)w\,dx=2\pi\int_{\R^N}w\gamma \log\gamma^2\,dx\qquad\forall w\in \left(T_\gamma\M\right)^\perp,
 	\end{split}
 	\]
 	combining the above identities we can rewrite, as $n\to\infty$,
 	\begin{equation}
 		\label{eq:espQ}
 	\begin{split}
 	Q(u_n)&\,=\frac{\varepsilon_n^2\|\nabla\phi_n\|_2^2-\pi(N-1)\varepsilon_n^2-\pi\int_{\R^N}R_{n}(x)\,dx+o(\varepsilon_n^2)}{\varepsilon_n^2+o(\varepsilon_n^2)}\\
 	&\,=(1+o(1))\left(\|\nabla\phi_n\|_2^2-\pi\int_{\R^N}\frac{R_{n}(x)}{\varepsilon_n^2}\,dx\right)-\pi(N-1)+o(1)\,,
 	\end{split}
 	\end{equation}
 	where
 	\begin{equation}
 	\label{eq:Rn}
 	R_{n}(x):=h(\gamma(x)+\varepsilon_n\phi_n(x))-h(\gamma(x))-\varepsilon_nh'(\gamma(x))\phi_n(x)\,.
 	\end{equation}
 	First, we study the behaviour of $R_n$ on the set
 	\begin{equation}
 		\label{eq:An}
 	A_n:=\left\{x\in\R^N\,:\,|\varepsilon_n\phi_n(x)|\leq\gamma(x)/2\right\}.
 	\end{equation}
 	Since $\gamma(x)/2\leq u_n(x)\leq3\gamma(x)/2$ for every $x\in A_n$, 
 	\[
 	\begin{split}
 	h''(\gamma(x)+t\varepsilon_n\phi_n(x))=&\,2\log((\gamma(x)+t\varepsilon_n\phi_n(x))^2)+6\\
 	= &\,4\log\left|\gamma(x)+t\varepsilon_n\phi_n(x)\right|+6
 	\leq 4\log\frac{3\gamma(x)}{2}+6=-2\pi|x|^2+4\log\frac32+6
 	\end{split}
 	\]
 	for every $t\in[0,1]$, every $x\in A_n$ and every $n\in\N$. Hence, by Taylor's formula with integral remainder, for every $x\in A_n$ we have
 	\[
 	R_n(x)=\varepsilon_n^2\phi_n(x)^2\int_0^1(1-t)h''(\gamma(x)+t\varepsilon_n\phi_n(x))\,dt\leq\varepsilon_n^2\phi_n(x)^2\left(-\pi|x|^2+2\log\frac32+3\right)
 	\]
 	that is
 	\begin{equation}
 		\label{eq:Rn_An}
 	-\frac{R_n(x)}{\varepsilon_n^2}\geq\phi_n(x)^2\left(\pi|x|^2-2\log\frac32-3\right)\,.
 	\end{equation}
 	Hence, 
 	\begin{equation}
 	\label{eq:intAn}
 	\begin{split}
 	-\int_{A_n}\frac{R_n(x)}{\varepsilon_n^2}\,dx&\,\geq\int_{A_n}\phi_n(x)^2\left(\pi|x|^2-2\log\frac32-3\right)\,dx\\
 	&\,\geq\pi\int_{A_n}|x|^2\phi_n(x)^2\,dx-\left(2\log\frac32+3\right)\int_{\R^N}\phi_n^2\,dx\\
 	&\,=\pi\int_{A_n}|x|^2\phi_n(x)^2\,dx -2\log\frac32-3
 	\end{split}
 	\end{equation}
 	uniformly in $n\in\N$. Note that the last integral is well defined for every $n$ because 
 	\[
 	\int_{A_n}|x|^2\phi_n(x)^2\,dx\leq\frac1{4\varepsilon_n^2}\int_{A_n}|x|^2\gamma(x)^2\,dx<\infty\,.
 	\]
 	Now we turn to $A_n^c$. Setting
 	\[
 	z(x):=\frac{\varepsilon_n\phi_n(x)}{\gamma(x)}\,,
 	\]
 	for every $x\in A_n^c$ it holds
 	\[
 	|z(x)|\geq \frac12\,.
 	\]
 	Moreover, using in particular that $\phi_n(x)\neq0$ on $A_n^c$, we can rewrite $\displaystyle R_n(x)/\varepsilon_n^2$ in terms of $z(x)$ as follows
 	\begin{equation}
 		\label{eq:R_Anc}
 	\begin{split}
 		-\frac{R_n(x)}{\varepsilon_n^2}=&\,-\frac{(\gamma+\varepsilon_n\phi_n)^2\log(\gamma+\varepsilon_n\phi_n)^2-\gamma^2\log\gamma^2-2\varepsilon_n\gamma\phi_n\left(\log\gamma^2+1\right)}{\varepsilon_n^2}\\
 		=&\,-\frac{\gamma^2\log\left(1+\frac{\varepsilon_n\phi_n}{\gamma}\right)^2+2\varepsilon_n\gamma\phi_n\log\left(1+\frac{\varepsilon_n\phi_n}{\gamma}\right)^2+\varepsilon_n^2\phi_n^2\log(\gamma+\varepsilon_n\phi_n)^2-2\varepsilon_n\gamma\phi_n}{\varepsilon_n^2}\\
 		=&-\phi_n^2\left(\frac{\log\left(1+z\right)^2}{z^2}+\frac{2\log\left(1+z\right)^2}{z}+\log\left(\gamma(1+z)\right)^2-\frac2z\right)\\
 		=&\,-\phi_n^2\left(\log(1+z)^2\left(\frac1{z^2}+\frac2z+1\right)+\log\gamma^2-\frac2z\right)\\
 		=&\,-\phi_n\left(-\pi|x|^2+\frac{(z+1)^2\log(1+z)^2}{z^2}-\frac2z\right)=\phi_n^2\left(\pi|x|^2-g(z(x))\right)
 	\end{split}
 \end{equation}
 	with
 	\[
 	g(z):=\frac{(1+z)^2\log(1+z)^2-2z}{z^2}\,.
 	\]
 	Note that there exists $C>0$ independent of $z$ such that
 	\begin{equation}
 	\label{eq:estg}
 	g(z)\leq2\log|z|+C\qquad\forall |z|>\frac12
 	\end{equation}
 	with the constant $2$ in front of $\log|z|$ being sharp. This is readily seen by writing
 	\[
 	g(z)-2\log|z|=2\left(\frac2z+\frac1{z^2}\right)\log|z|+2\left(1+\frac1z\right)^2\log\left|1+\frac1z\right|-\frac2z
 	\]
 	and observing that the right hand side is a continuous function on $|z|\geq1/2$ that tends to $0$ as $|z|\to\infty$. Combining \eqref{eq:R_Anc} and \eqref{eq:estg} then yields
 	\begin{equation}
 		\label{eq:Rn_Anc}
 	\begin{split}
 	-\frac{R_n(x)}{\varepsilon_n^2}&\,\geq \phi_n(x)^2\left(\pi|x|^2-2\log\left|\frac{\varepsilon_n\phi_n(x)}{\gamma(x)}\right|-C\right)\\
 	&\,=\phi_n(x)^2\left(-2\log\varepsilon_n-2\log|\phi_n(x)|-C\right)\\
 	&\,\geq2|\log\varepsilon_n|\phi_n(x)^2-2\phi_n(x)^2\log_+|\phi_n(x)|-C\phi_n(x)^2\,,
 	\end{split}
 	\end{equation}
 	where we used that $\varepsilon_n\in(0,1)$ for every $n\in\N$. Hence, 
 	\begin{equation}
 		\label{eq:intAnc}
 	-\int_{A_n^c}\frac{R_n(x)}{\varepsilon_n^2}\,dx\geq2|\log\varepsilon_n|\int_{A_n^c}\phi_n(x)^2\,dx -2\int_{\R^N}\phi_n^2\left(\log_+|\phi_n|+\frac C2\right)\,dx\,.
 	\end{equation}
 	However, since for every $\delta>0$ there exists $C_\delta>0$ such that 
 	\begin{equation}
 	\label{eq:log<pot}
 	\phi_n(x)^2\log_+|\phi_n(x)|\leq C_\delta |\phi_n(x)|^{2+\delta}\qquad\forall x\in \R^N
 	\end{equation}
 	for every $n$, by the Gagliardo-Nirenberg inequality
 	\begin{equation}
 		\label{eq:GN}
 		\|w\|_p^p\leq K_p\|w\|_2^{p- N\left(\frac p2-1\right)}\|\nabla w\|_2^{N\left(\frac p2-1\right)}\qquad\forall w\in H^1(\mathbb{R}^N)
 	\end{equation}
 	one has
 	\begin{equation}
 		\label{eq:pospartGN}
 	\int_{\R^N}\phi_n^2\log_+|\phi_n|\,dx\leq K_\delta\|\phi_n\|_2^{2+\delta-\frac{N\delta}2}\|\nabla\phi_n\|_2^{\frac{N\delta}{2}}=K_\delta \|\nabla\phi_n\|_2^{\frac{N\delta}{2}}
 	\end{equation}
 	for a suitable constant $K_\delta>0$ depending on $\delta$ only. Hence, fixing $\delta$ such that $N\delta/2<2$ and combining \eqref{eq:espQ}, \eqref{eq:intAn} and \eqref{eq:intAnc} leads to
 	\[
 	\begin{split}
 	Q(u_n)+o(1)\geq&\, (1+o(1))\|\nabla\phi_n\|_2^2-\pi(N-1)-\pi\left(2\log\frac32+3+C\right)\\
 	&+(1+o(1))\left(\pi^2\int_{A_n}|x|^2\phi_n(x)^2\,dx+2\pi|\log\varepsilon_n|\int_{A_n^c}\phi_n(x)^2\,dx-2\pi K_\delta\|\nabla\phi_n\|_2^{\frac{N\delta}{2}}\right).
 	\end{split}
 	\]
 	The uniform boundedness of $Q(u_n)$ and the fact that $\varepsilon_n\to0$ as $n\to\infty$ then yields the boundedness of $(\phi_n)_n$ in $H^1(\R^N)$ and 
 	\begin{equation}
 		\label{eq:An_Anc}
 		\sup_n\int_{A_n}|x|^2\phi_n(x)^2\,dx<\infty\qquad\text{and}\qquad\lim_{n}\int_{A_n^c}\phi_n(x)^2\,dx=0\,.
 	\end{equation}
 	Up to subsequences, $\phi_n\rightharpoonup \phi$ in $H^1(\R^N)$ as $n\to\infty$, for some $\phi\in H^1(\R^N)$, and to complete the proof we are left to show that the convergence is strong in $L^2(\R^N)$ (the fact that $\phi\not\equiv0$ and that it satisfies \eqref{eq:lim_ortM} will then directly follow by \eqref{eq:phiL2=1} and \eqref{eq:phi_ortT}). 
 	
 	To this end, it is enough to show that $(\phi_n^2)_n$ is tight in $L^1(\R^N)$, i.e. that for every $\eta>0$ there exists $R>0$ such that 
 	\begin{equation}
 	\label{eq:tight}
 	\sup_n\int_{B_R^c}\phi_n(x)^2\,dx<\eta\,.
 	\end{equation}
 	Now, by the first estimate in \eqref{eq:An_Anc}, for any fixed $R>0$ we have
 	\[
 	\int_{B_R^c\cap A_n}\phi_n(x)^2\,dx\leq\frac1{R^2}\int_{B_R^c\cap A_n}|x|^2\phi_n(x)^2\,dx\leq\frac C{R^2}
 	\]
 	uniformly in $n$, for some $C>0$. Hence, for every $\eta>0$ there exists $R>0$ such that
 	\begin{equation}
 	\label{eq:tightAn}
 	\sup_n\int_{B_R^c\cap A_n}\phi_n(x)^2\,dx<\frac\eta2\,.
 	\end{equation}
 	Conversely, assume by contradiction that there exists $\overline\eta>0$ so that for every $R>0$ one has
 	\[
 	\sup_n\int_{B_R^c\cap A_n^c}\phi_n(x)^2\,dx>\frac{\overline{\eta}}2\,.
 	\]
 	Then there exists a sequence $R_k\to\infty$ and a subsequence $(\phi_{n_k})_k$ for which 
 	\[
 	\liminf_k\int_{B_R^c\cap A_n^c}\phi_{n_k}(x)^2\,dx\geq\frac{\overline\eta}{2}
 	\]
 	which is impossible by the second estimate in \eqref{eq:An_Anc}. Hence, for every $\eta>0$ there exists $R>0$ so that
 	\[
 	\sup_n\int_{B_R^c\cap A_n^c}\phi_n(x)^2\,dx<\frac\eta2\,,
 	\]
 	that together with \eqref{eq:tightAn} gives \eqref{eq:tight} and concludes the proof.
 	\end{proof}
 
 \begin{proof}[Proof of Proposition \ref{prop:BE}]
 	Let $(u_n)_n\subset H^1(\R^N)$ satisfy \eqref{eq:hyp_BE}. For every $n$, by Remark \ref{rem:d_assunta} we have $\displaystyle d(u_n,\M)=d(u_n, a_n\gamma(\cdot-b_n))$ with $a_n^2=m(u_n)$ and for some $b_n\in\R^N$. Since $0\leq m(u_n)\leq \|u_n\|_2^2$ by definition, $(a_n^2)_n$ is bounded from above. Moreover, since by assumption $d(u_n,\M)\to0$ as $n\to\infty$, we also have that $(a_n^2)_n$ is bounded away from zero, since if this were not the case by \eqref{eq:hyp_BE} it would follow $d(u_n,\M)^2=\|u_n\|_2^2-m(u_n)=\|u_n\|_2^2+o(1)>0$ for large $n$. By the invariance of $Q$ under translations and multiplication by a constant, with no loss of generality (and possibly up to subsequences) we can thus assume $a_n=1$ and $b_n=0$ for every $n$, so that \eqref{eq:hyp_BE} remains valid and $d(u_n,\M)=d(u_n,\gamma)$ for every $n$. Therefore, there exist $\varepsilon_n\to0$ and $(\phi_n)_n\subset H^1(\R^N)\cap\left(T_\gamma\M\right)^\perp$ such that $\|\phi_n\|_2=1$ for every $n$ and 
 	\[
 	u_n=\gamma+\varepsilon_n\phi_n\qquad\forall n.
 	\]
 	By \eqref{eq:hyp_BE} and Lemma \ref{lem:boundBE}, $(\phi_n)_n$ is bounded in $H^1(\R^N)$ and, up to subsequences, $\phi_n\rightharpoonup\phi$ in $H^1(\R^N)$ and $\phi_n\to\phi$ in $L^2(\R^N)$ as $n\to\infty$, for some $\phi\in H^1(\R^N)\cap\left(T_\gamma\M\right)^\perp$ with $\|\phi\|_2=1$. Recalling the expansion in the first line of \eqref{eq:espQ} and noting that, by weak lower semicontinuity,
 	\begin{equation}
 		\label{eq:phintophi1}
 		\liminf_n\frac{\varepsilon_n^2\|\nabla \phi_n\|_2^2-\pi(N-1)\varepsilon_n^2}{\varepsilon_n^2+o(\varepsilon_n^2)}\geq\|\nabla\phi\|_2^2-\pi(N-1)\,,
 	\end{equation}
 	we need to take into account the term $\displaystyle-\pi\int_{\R^N}R_n(x)/\varepsilon_n^2\,dx$, with $R_n$ as in \eqref{eq:Rn}. 
 	
 	We consider first the set $B_R^c$ for fixed $\displaystyle R>\overline R:=\sqrt{\frac1\pi\left(2\log\frac32+3\right)}$. By \eqref{eq:Rn_An}, it holds
 	\begin{equation}
 	\label{eq:Rn>0}
 	-\frac{R_n(x)}{\varepsilon_n^2}\geq0\qquad\forall x\in A_n\cap B_R^c
 	\end{equation}
 	for every $n$, where $A_n$ is the set defined in \eqref{eq:An}. Conversely, by \eqref{eq:Rn_Anc} it follows 
 	\begin{equation}
 	\label{eq:Rn_code2}
 	-\frac{R_n(x)}{\varepsilon_n^2}\geq-\phi_n^2(x)\left(\log_+|\phi_n(x)|+C\right)\qquad\forall n\in A_n^c\cap B_R^c\,.
 	\end{equation}
 	Observe however that
 	\begin{equation}
 		\label{eq:unif1}
 		\lim_{R\to\infty}\limsup_n\int_{B_R^c}\phi_n^2(x)\left(\log_+|\phi_n(x)|+C\right)\,dx=0\,.
 	\end{equation}
 	Indeed, since $(\phi_n)_n$ is bounded in $H^1(\R^N)$ and $\phi_n\to\phi$ in $L^2(\R^N)$, then $\phi_n\to\phi$ in $L^p(\R^N)$ for every $p\geq2$ if $N=1,2$ and every $p\in[2,2^*)$ if $N\geq3$. Fixing any such $p$, by \eqref{eq:pospartGN} there exists $C_p>0$ such that 
 	\[
 	\int_{B_R^c}\phi_n^2(x)\log_+|\phi_n(x)|\,dx\leq C_p\int_{B_R^c}|\phi_n|^p\,dx\,,
 	\]
 	and by the strong convergence of $\phi_n$ to $\phi$ in $L^p(\R^N)$
 	\[
 	\|\phi_n\|_{L^p(B_R^c)}\leq\|\phi_n-\phi\|_{L^p(B_R^c)}+\|\phi\|_{L^p(B_R^c)}=\|\phi\|_{L^p(B_R^c)}+o(1)\qquad\text{as }n\to\infty\,,
 	\]
 	thus yielding
 	\[
 	\lim_{R\to\infty}\limsup_n\int_{B_R^c}\phi_n^2(x)\log_+|\phi_n(x)|\,dx=0\,.
 	\]
 	Since the term $\displaystyle\int_{B_R^c}\phi_n^2\,dx$ can be handled in the same way, we obtain \eqref{eq:unif1}. 
 	Combining \eqref{eq:Rn>0}, \eqref{eq:Rn_code2} and \eqref{eq:unif1} we have
 	\begin{equation}
 		\label{eq:code}
 		\lim_{R\to\infty}\liminf_n\int_{B_R^c}-\frac{R_n(x)}{\varepsilon_n^2}\,dx\geq0\,.
 	\end{equation} 
 	Now, on $B_R$ we prove that
 	\begin{equation}
 		\label{eq:Rn_BR}
 		\lim_n\int_{B_R}-\frac{R_n(x)}{\varepsilon_n^2}\,dx=\int_{B_R}(\pi|x|^2-3)\phi(x)^2\,dx\,.
 	\end{equation}
 	To this end, note first that \eqref{eq:R_Anc} and the fact that there exists $M>0$ such that $g(z)\geq-M$ for every $\displaystyle|z|\geq\frac12$ give
 	\[
 	-\frac{R_n(x)}{\varepsilon_n}\leq \phi_n(x)^2(\pi|x|^2+M)\qquad\forall x\in A_n^c\,,
 	\]
 	that together with \eqref{eq:Rn_code2} yields
 	 \[
 	\left|\frac{R_n(x)}{\varepsilon_n^2}\right|\leq C_R\phi_n(x)^2\left(\log_+|\phi_n(x)|+1\right)\qquad\forall x\in B_R\cap A_n^c
 	\]
 	for a suitable constant $C_R>0$ depending on $R$ only. Hence, by \eqref{eq:log<pot} and H\"older inequality,
 	\begin{equation}
 		\label{eq:int<mis}
 	\begin{split}
 	\int_{B_R\cap A_n^c}\left|\frac{R_n(x)}{\varepsilon_n^2}\right|\,dx&\,\leq C_{R,\delta}\int_{B_R\cap A_n^c}\phi_n^2+|\phi_n|^{2+\delta}\,dx\\
 	&\,\leq C_{R,\delta}\left(\left|B_R\cap A_n^c\right|^{\alpha_1}+\left|B_R\cap A_n^c\right|^{\alpha_2}\right)\|\phi_n\|_{L^{2+\tilde\delta}(\R^N)}^\beta\\
 	&\,\leq  C_{R,\delta}'\left(\left|B_R\cap A_n^c\right|^{\alpha_1}+\left|B_R\cap A_n^c\right|^{\alpha_2}\right)
 	\end{split}
 	\end{equation}
 	for suitable $0<\delta<\tilde\delta$, $\alpha_1,\alpha_2,\beta>0$ and constants $C_{R,\delta}, C_{R,\delta}'>0$ depending on $R$ and $\delta$ but not on $n$. This immediately ensures that $\displaystyle \int_{B_R\cap A_n^c}\left|\frac{R_n(x)}{\varepsilon_n^2}\right|\,dx$ is bounded uniformly in $n$. Furthermore, since on $B_R\cap A_n^c$ we have $\displaystyle |\phi_n(x)|\geq\frac{\gamma(x)}{2\varepsilon_n}\geq\frac{e^{-\pi R^2/2}}{2\varepsilon_n}$,
 	\[
 	1\geq\int_{B_R\cap A_n^c}|\phi_n|^2\,dx\geq \frac{e^{-\pi R^2}}{4\varepsilon_n^2}\left|B_R\cap A_n^c\right|\,,
 	\] 
 	which shows that $\displaystyle\left|B_R\cap A_n^c\right|\to0$ as $n\to\infty$. Coupling with \eqref{eq:int<mis} entails
 	\begin{equation}
 		\label{eq:negl}
 		\lim_n\int_{B_R\cap A_n^c}\left|\frac{R_n(x)}{\varepsilon_n^2}\right|\,dx=0\,.
 	\end{equation}
	On the contrary, on $B_R\cap A_n$ we have $\varepsilon_n|\phi_n|\leq\gamma/2$ and $\|\phi_n\|_2=1$ for every $n$, so  $\varepsilon_n\phi_n\to0$ in $L^2(B_R)$, and thus also a.e. on $B_R$ up to subsequences. All in all, exploiting once more Taylor's formula with integral remainder for $h(s)=s^2\log s^2$
	\[
	-\frac{R_n(x)}{\varepsilon_n^2}=-\phi_n(x)^2\int_0^1(1-t)h''(\gamma(x)+t\varepsilon_n\phi_n(x))\,dt\,,
	\]
	since $\gamma$ is bounded away from zero on $B_R$ we have that the integral term on the right hand side is uniformly bounded on $B_R\cap A_n$ and for a.e. $x$ in this set
	\[
	-\int_0^1(1-t)h''(\gamma(x)+t\varepsilon_n\phi_n(x))\,dt\to -\frac12 h''(\gamma(x))=\pi|x|^2-3\qquad\text{as }n\to\infty\,.
	\]
	Together with the strong convergence in $L^2(\R^N)$ of $\phi_n$ to $\phi$ and \eqref{eq:negl} (and possibly passing to a further subsequence), this yields \eqref{eq:Rn_BR}.
	
	We are now in position to complete the proof of Proposition \ref{prop:BE}. Combining \eqref{eq:espQ}, \eqref{eq:phintophi1}, \eqref{eq:code} and \eqref{eq:Rn_BR}, for every $R>0$ and for sufficiently large $n$ we have
	\[
	Q(u_n)\geq \|\nabla\phi\|_2^2-\pi(N-1)+\pi^2\int_{B_R}|x|^2\phi(x)^2\,dx-3\pi\int_{B_R}\phi^2\,dx+\eta(R)+o(1)
	\]
	with $\eta(R)\to0$ as $R\to\infty$, so that taking $R\to\infty$ (and recalling also that $\|\phi\|_2=1$)
	\[
	\liminf_n Q(u_n)\geq \|\nabla\phi\|_2^2+\pi^2\int_{\R^N}|x|^2\phi(x)^2\,dx-\pi(N+2)\,.
	\]
	As a byproduct, this shows that $\displaystyle\int_{\R^N}|x|^2\phi(x)^2\,dx<\infty$. Furthermore, since $\phi\in \left(T_\gamma\M\right)^\perp$ and $\|\phi\|_2=1$, it is well known (see e.g. \cite{thangavelu}) that 
	\[
	\|\nabla\phi\|_2^2+\pi^2\int_{\R^N}|x|^2\phi(x)^2\,dx\geq \pi(N+4)\,,
	\]
	which plugged in the previous estimates gives \eqref{eq:Q>2} and concludes the proof.
 \end{proof}

	\section{Existence of optimizers in $H^1(\R^N)$: proof of Theorem \ref{thm:main}}
	\label{sec:comp}
	
	This section is devoted to complete the proof of Theorem \ref{thm:main}. Since the upper bound $Q<2\pi\log2$ has already been proved in Section \ref{sec:Q<2log2}, we are left to show that $Q$ is attained by some $u\in H^1(\R^N)$. 
	
	To this end, let $(u_n)_n\subset H^1(\R^N)$ be such that 
	\[
	Q(u_n)\to Q\qquad\text{as }n\to\infty\,.
	\]
	By the invariances of $Q(\cdot)$, with no loss of generality we can further assume that, for every $n$,
	\begin{equation}
	\label{eq:norm_un}
	\|u_n\|_2=1\,,\qquad \sup_{j\in\Z^N}\int_{A_j}u_n^2\,dx=\int_{A_0}u_n^2\,dx\,,
	\end{equation}
	where $A_j:=[0,1)^N+j$, for every $j\in\Z^N$.
	
	Since $Q<2\pi\log2$, by \eqref{eq:norm_un}, Remark \ref{rem:d_assunta} and Section \ref{sec:estM} it directly follows that there exists $\alpha>0$ such that 
	\[
	1>d(u_n,\M)\geq\alpha>0\qquad\forall n\,.
	\] 
	Moreover, arguing similarly to the proof of Lemma \ref{lem:boundBE} it is easily seen that, given any $\delta>0$ such that $\displaystyle N\delta/2<2$, 
	\[
	Q(u_n)+o(1)\geq \|\nabla u_n\|_2^2-\pi K_\delta\|\nabla u_n\|_2^{\frac{N\delta}{2}}-N\pi
	\]
	as $n\to\infty$, which shows that $(u_n)_n$ is bounded in $H^1(\R^N)$. Up to subsequences (not relabeled), we then have that $u_n\rightharpoonup u$ in $H^1(\R^N)$ as $n\to\infty$. 
	
	We first prove that $u\not\equiv0$. Indeed, if it were $u\equiv0$, then by strong convergence in $L_{loc}^2(\R^N)$ we would have $\displaystyle\lim_n \int_{A_0}u_n^2\,dx=0$, in turn yielding $u_n\to0$ in $L^p(\R^N)$ for every $p\in(2,2^*)$ if $N\geq3$ and every $p>2$ if $N=1,2$ by the standard vanishing lemma based on local Gagliardo--Nirenberg estimates on the cubes $A_j$. Hence, recalling once more that $u_n^2\log_+ u_n^2\leq C_p|u_n|^p$ on $\R^N$, for a suitable constant $C_p>0$ depending only on $p>2$, we obtain
	\begin{equation}
		\label{eq:pospart}
	\lim_n\int_{\R^N}u_n^2\log_+ u_n^2\,dx=0\,.
	\end{equation}
	On the other hand, for every fixed $\delta\in (0,1)$, 
	\begin{equation}
	\label{eq:negpart}
	\int_{\R^N}u_n^2\log_- u_n^2\,dx=\int_{\left\{|u_n|\leq1\right\}}u_n^2\log u_n^2\,dx\leq \int_{\left\{|u_n|\leq\delta\right\}}u_n^2\log u_n^2\,dx\leq \log\delta^2\int_{\left\{|u_n|\leq\delta\right\}}u_n^2\,dx\,.
	\end{equation}
	Observe now that, if there were any $\overline\delta\in(0,1)$ for which $\displaystyle \int_{\left\{|u_n|\leq\overline{\delta}\right\}}u_n^2\,dx<\frac12$ for every $n$, then we would have $\displaystyle\int_{\left\{|u_n|\geq\overline\delta\right\}}u_n^2\,dx\geq\frac12$ for every $n$. This is however impossible, since for every $p>2$ there exists $C_{p,\overline{\delta}}>0$ depending only on $p$ such that $s^p\geq C_{p,\overline\delta}s^2$ whenever $s\geq\overline\delta$, so that $\displaystyle\int_{\left\{|u_n|\geq\overline\delta\right\}}u_n^p\,dx\geq C_{p,\overline\delta}\int_{\left\{|u_n|\geq\overline\delta\right\}}u_n^2\,dx\geq\frac12$ for any given $p\in(2,2^*)$ if $N\geq3$ or $p>2$ if $N=1,2$, contradicting the strong convergence of $u_n$ to $0$ in $L^p(\R^N)$. Therefore, for every $\delta\in(0,1)$ there exists $n_\delta\in\N$ such that
	\[
	\int_{\left\{|u_{n_\delta}|\leq\overline{\delta}\right\}}u_{n_\delta}^2\,dx\geq\frac12
	\]
	and $n_\delta\to+\infty$ as $\delta\to0$. Combining with \eqref{eq:pospart} and \eqref{eq:negpart} this entails (at least along a subsequence)
	\[
	\lim_n\int_{\R^N}u_n^2\log u_n^2\,dx=-\infty\,,
	\]
	which is impossible since $Q(u_n)\to Q$, $(u_n)_n$ is bounded in $H^1(\R^N)$ and $\|u_n\|_2=1$ for every $n$. Hence, $u\not\equiv0$ on $\R^N$ and, by weak lower semicontinuity, 
	\begin{equation}
		\label{eq:semicont}
		0<\|u\|_2\leq\liminf_n\|u_n\|_2\,.
	\end{equation}
	
	\begin{proposition}
		\label{prop:convL2}
		If $u_n\to u$ in $L^2(\R^N)$ as $n\to+\infty$, then $Q(u)=Q$.
	\end{proposition}
	\begin{proof}
		It is enough to show that the strong convergence of $u_n$ to $u$ in $L^2(\R^N)$ implies that
		\begin{equation}
		\label{eq:us_ent}
		\int_{\R^N}u^2\log u^2\,dx\geq\limsup_n 	\int_{\R^N}u_n^2\log u_n^2\,dx\,.
		\end{equation}
	Indeed, if \eqref{eq:us_ent} holds, since the strong convergence in $L^2(\R^N)$ implies $\displaystyle d(u_n,\M)\to d(u,\M)$ (which in turn ensures that $d(u,\M)\geq\alpha>0$, so $u\not\in\M$), by \eqref{eq:us_ent} and the weak lower semicontinuity of $\|\nabla u_n\|_2$ one obtains
	\[
	Q=\lim_n Q(u_n)\geq Q(u)\geq Q\,.
	\]
	We are thus left to prove the validity of \eqref{eq:us_ent}. We write 
	\[
	\int_{\R^N}u_n^2\log u_n^2\,dx=\int_{\R^N}u_n^2\log_+u_n^2\,dx+\int_{\R^N}u_n^2\log_-u_n^2\,dx
	\]
	and observe that the first integral on the right hand side is non negative, while the second one is non-positive. Since $(u_n)_n$ is bounded in $H^1(\R^N)$ and $u_n\to u$ in $L^2(\R^N)$, the Gagliardo-Nirenberg inequality \eqref{eq:GN} implies that $u_n\to u$ in $L^p(\R^N)$ for every $p\in(2,2^*)$ if $N\geq3$ and $p>2$ if $N=1,2$. Hence, given $s>0$ there exists $g\in L^1(\R^N)$ such that (up to subsequences) $|u_n|^{2+s}\leq g$ a.e. on $\R^N$. Since $u_n^2\log_+u_n^2\leq C_s|u_n|^{2+s}$ everywhere on $\R^N$ for a suitable constant $C_s>0$ independent of $n$, we then have that $u_n^2\log_+u_n^2$ is bounded from above a.e. in $\R^N$ by some fixed function  in $L^1(\R^N)$ uniformly on $n$. Since the strong convergence of $u_n$ to $u$ in $L^2(\R^N)$ also implies (again up to subsequences) the pointwise convergence a.e. of $u_n^2\log_+u_n^2$ to $u^2\log_+u^2$, by dominated convergence we obtain
	\begin{equation}
	\label{eq:conv+}
	\int_{\R^N}u_n^2\log_+u_n^2\,dx\to\int_{\R^N}u^2\log_+u^2\,dx\qquad\text{as }n\to+\infty\,.
	\end{equation}
	Similarly, since up to subsequences $u_n^2\log_-u_n^2\to u^2\log_-u^2$ a.e. on $\R^N$ and $u_n^2\log_-u_n^2\leq0$ everywhere, by Fatou's lemma (that can be used here because $(u_n^2\log_-u_n^2)_n$ is uniformly bounded in $L^1(\R^N)$ by the boundedness of $Q(u_n)$ and the boundedness of $(u_n)_n$ in $H^1(\R^N)$)
	\[
	\begin{split}
	\limsup_n\int_{\R^N}u_n^2\log_-u_n^2\,dx=&\,-\liminf_n\int_{\R^N}\left|u_n^2\log_- u_n^2\right|\,dx\\
	\leq&\,-\int_{\R^N}\left|u^2\log_- u^2\right|\,dx=\int_{\R^N}u^2\log_-u^2\,dx\,.
	\end{split}
	\]
	Combining with \eqref{eq:conv+} gives \eqref{eq:us_ent}. Note that, even though the previous argument relies on the extraction of subsequences, since by assumption the whole sequence $u_n$ converges to $u$ in $L^2(\R^N)$, from every subsequence one can extract a further subsequence for which \eqref{eq:us_ent} holds, thus guaranteeing that \eqref{eq:us_ent} is actually true for the whole sequence.
	\end{proof}
	In view of \eqref{eq:semicont} and Proposition \ref{prop:convL2}, to complete the proof of Theorem \ref{thm:main} we need to rule out the case
	\begin{equation}
		\label{eq:assurdo}
		\|u\|_2<\liminf_n\|u_n\|_2\,.
	\end{equation} 
	To do this, we will follow closely the strategy developed by K\"onig in \cite{Konig_JEMS} for the Sobolev inequality. This requires to establish the next lemmas. 
	\begin{lemma}
		\label{lem:m=max}
		As $n\to +\infty$, it holds
		\[
		m(u_n)
		=
		\max\{m(u),m(u_n-u)\}+o(1)
		\]
		where $m(\cdot)$ is the quantity defined in Remark \ref{rem:d_assunta}.
	\end{lemma}
	
	\begin{proof}
		By Remark \ref{rem:d_assunta}, let $b_n\in\mathbb{R}^N$ be such that
		\[
		m(u_n)
		=
		\left(\int_{\mathbb{R}^N}
		u_n(x)e^{-\frac{\pi}{2}|x-b_n|^2}\,dx\right)^2.
		\]
		If $(b_n)_n$ is bounded, there exists $b_\infty\in\mathbb R^N$ such that, up to subsequences,
		\[
		b_n\to b_\infty\in\mathbb{R}^N,
		\]
		whence
		\[
		\int_{\mathbb{R}^N}
		(u_n-u)e^{-\frac{\pi}{2}|x-b_n|^2}\,dx\to0\qquad\text{as }n\to\infty
		\]
		and therefore
		\[
		\begin{split}
			m(u_n)
			=&\,
		\left(\int_{\mathbb{R}^N}
			(u_n-u)e^{-\frac{\pi}{2}|x-b_n|^2}\,dx
			+
			\int_{\mathbb{R}^N}
			u e^{-\frac{\pi}{2}|x-b_n|^2}\,dx\right)^2\\
			=&\,\left(
			\int_{\mathbb{R}^N}
			u e^{-\frac{\pi}{2}|x-b_\infty|^2}\,dx\right)^2
			+o(1)
			\le m(u)+o(1).
		\end{split}
		\]
		On the other hand, if $(b_n)_n$ is unbounded, then
		\[
		\int u\,e^{-\frac{\pi}{2}|x-b_n|^2}\,dx\to 0,
		\]
		and arguing as above we obtain
		\[
		m(u_n)\le m(u_n-u)+o(1),
		\]
		which yields
		\[
		m(u_n)\le
		\max\{m(u),m(u_n-u)\}+o(1).
		\]
		Conversely, denoting by $\overline b\in\mathbb{R}^N$ the vector at which $m(u)$ is attained (whose existence is ensured again by Remark \ref{rem:d_assunta}), one has
		\[
		\begin{split}
		m(u)
		=&\,
		\left(\int_{\mathbb{R}^N} u\,e^{-\frac{\pi}{2}|x-\overline b|^2}\,dx\right)^2\\
		=&\,
		\left(\int_{\mathbb{R}^N} (u-u_n)e^{-\frac{\pi}{2}|x-\overline b|^2}\,dx
		+
		\int_{\mathbb{R}^N} u_n e^{-\frac{\pi}{2}|x-\overline b|^2}\,dx\right)^2
		\le m(u_n)+o(1).
		\end{split}
		\]
		Similarly, since
		\[
		m(u_n-u)
		=
		\left(\int (u_n-u)e^{-\frac{\pi}{2}|x-\tilde b_n|^2}\,dx\right)^2
		\]
		for suitable $\tilde b_n$, if $(\tilde b_n)_n$ is unbounded, then arguing as above we obtain
		\[
		m(u_n-u)\le m(u_n)+o(1),
		\]
		whereas, if $(\tilde b_n)_n$ is bounded, one has
		\[
		m(u_n-u)=o(1).
		\]
		Combining the preceding estimates gives the conclusion.
	\end{proof}
\begin{remark}
	\label{rem:BL}
	In the following, we will need identities in the form
	\begin{equation}
		\label{eq:BL}
		\int_{\R^N}\left|h(w_n)-h(w)-h(w_n-w)\right|\,dx=o(1)\qquad\text{as }n\to+\infty
	\end{equation}
	where $h(s):=s^2\log s^2$, $w_n\rightharpoonup w$ in $H^1(\R^N)$ as $n\to\infty$, and $\displaystyle \sup_n \int_{\R^N}|h(w_n)|\,dx<\infty$. Specifically, we will use this relation with $w_n$ given by linear combinations of the minimizing sequence $u_n$ we are considering and of its weak limit $u$ (up to subsequences), i.e.
	\begin{equation}
		\label{eq:genw}
		w_n=A_n u_n+ B_n u
	\end{equation}
	for suitable coefficients $A_n, B_n\in \R$ uniformly bounded in $n$. The validity of \eqref{eq:BL} is a direct application of the general case of the Brezis-Lieb lemma \cite[Theorem 2]{BL2}, of which we now check the hypotheses in our setting. Observe first that $h:\R\to\R$ is a continuous function with $h(0)=0$. Moreover, for every $\varepsilon>0$ there exists $C_\varepsilon>0$ such that
	\begin{equation}
		\label{eq:condBL}
		|h(x+y)-h(x)|\leq \varepsilon\Phi(x)+C_\varepsilon\Phi(y)\qquad\forall x,y\in\R
	\end{equation}
	where $\Phi:\R\to[0,\infty)$ is the convex envelope of the function $\tilde\Phi(s):=s^2\left(|\log s^2|+1\right)$, $s\in\R$. By definition, $\Phi$ is convex on $\R$, even, $\Phi\equiv\tilde\Phi$ on $[0,+\infty)\setminus(a,1)$ for some $a>0$, and both $\Phi$ and $\tilde{\Phi}$ are continuous on $[a,1]$. Hence, there exist two constants $c_1,c_2>0$ such that
	\begin{equation}
		\label{eq:equiPhi}
	c_1\tilde\Phi(s)\leq \Phi(s)\leq c_2\tilde\Phi(s)\qquad\forall s\in\R\,.
	\end{equation}
	To obtain \eqref{eq:condBL} it is then enough to prove the estimate with $\tilde\Phi$ in place of $\Phi$ on the right hand side. This can be done by noting that for every $A>0$ there exists $C_A>0$ such that 
	\begin{equation}
		\label{eq:dilPhi}
		\tilde\Phi(As)\leq C_A\tilde\Phi(s)
	\end{equation}
	(take e.g. $C_A:=A^2(2+|\log A^2|$). For any fixed $\delta>0$, if $|y|\geq\delta|x|$, then $|x+y|\leq (1+\delta^{-1})|x|$ and by \eqref{eq:dilPhi} we have
	\begin{equation}
		\label{eq:estPhi}
		|h(x+y)-h(x)|\leq|h(x+y)|+|h(x)|=|h(|x+y|)|+|h(|x|)|\leq \left(C_{\delta^{-1}}+C_{1+\delta^{-1}}\right)\tilde\Phi(|y|)
	\end{equation}
	where we also used the fact that, by definition, $|h(s)|\leq \tilde\Phi(s)$ for every $s$. Conversely, if $|y|<\delta|x|$, we can write
	\[
	|h(x+y)-h(x)|\leq |y|\sup_{\theta\in[0,1]}|h'(x+\theta y)|\leq \delta|x|\sup_{\theta\in[0,1]}|h'(x+\theta y)|\,,
	\]
	and since $(1-\delta)|x|\leq|x+\theta y|\leq (1+\delta)|x|$ for every $\theta\in(0,1)$ and $h'(s)=2s(\log s^2+1)$, we have
	\[
	\begin{split}
		|h'(x+\theta y)|\leq &\,2|x+\theta y|\left(|\log(x+\theta y)^2|+1\right)\\
		\leq&\, 2(1+\delta)|x|\left(2|\log x^2|+|\log(1-\delta)^2|+|\log(1+\delta)^2|+1\right)\leq 8|x|\left(|\log x^2|+1\right)
	\end{split}
	\]
	for sufficiently small $\delta$. Hence, 
	\[
	|h(x+y)-h(x)|\leq8\delta\tilde\Phi(x)\qquad\text{if }|y|<\delta|x|\,,
	\]
	so that taking $\delta$ small enough and combining with \eqref{eq:equiPhi} and \eqref{eq:estPhi} gives \eqref{eq:condBL}. Therefore, $h(\cdot)$ satisfies \cite[Eq. (3)]{BL2}, so to apply \cite[Theorem 2]{BL2} to the sequence $(w_n)_n$ as in \eqref{eq:genw} we are left to verify that conditions (i)--(iv) of that result are satisfied along such sequence. To see this, recall that $(u_n)_n$ is a minimizing sequence for $Q(\cdot)$ in $H^1(\R^N)$. As such, we already observed that it is bounded in $H^1(\R^N)$, which together with the Gagliardo-Nirenberg inequalities \eqref{eq:GN} and the uniform boundedness of $Q(u_n)$ and of $d(u_n,\M)$ implies that $h(u_n)$ is uniformly bounded in $L^1(\R^N)$. Since, up to subsequences, $u_n\to u$ a.e., by Fatou's lemma it then follows that $h(u)\in L^1(\R^N)$ too. Together with the uniform boundedness in $H^1(\R^N)$, this also yields $\Phi(u_n)$ uniformly bounded in $L^1(\R^N)$ and $\Phi(u)\in L^1(\R^N)$. Since by definition $w_n$ is a linear combination of $u_n$ and $u$ with uniformly bounded coefficients, so is $w_n-w$. Hence, by the convexity of $\Phi$, \eqref{eq:equiPhi} and \eqref{eq:dilPhi}, it follows that $\Phi(w_n), \Phi(w_n-w)$ are uniformly bounded in $L^1(\R^N)$, and $\Phi(w)\in L^1(\R^N)$.  This shows that (iii)--(iv) in \cite[Theorem 2]{BL2} are satisfied, and (ii) follows by the inequality $|h(s)|\leq\tilde\Phi(s)$ for every $s\in\R$. 
\end{remark}
	\begin{lemma}
		\label{lem:mn=m}
		If
		\[
			\|u\|_2<\liminf_n\|u_n\|_2\,.
		\]
		then
		\[
		m(u_n-u)=m(u)+o(1).
		\]
	\end{lemma}
	\begin{proof}
		We argue by contradiction and assume first that (up to subsequences)
		\begin{equation}
			\label{mn>m}
			\lim_n m(u_n-u)>m(u)\,.
		\end{equation}
		 From the weak convergence of $u_n$ to $u$, we have
		\[
		\int_{\mathbb{R}^N}|\nabla u_n|^2\,dx
		=
		\int_{\mathbb{R}^N}|\nabla u|^2\,dx
		+
		\int_{\mathbb{R}^N}|\nabla(u_n-u)|^2\,dx
		+o(1),
		\]
		\[
		\int_{\mathbb{R}^N}u_n^2\,dx
		=
		\int_{\mathbb{R}^N}u^2\,dx
		+
		\int_{\mathbb{R}^N}(u_n-u)^2\,dx
		+o(1),
		\]
		and by Remark \ref{rem:BL} (with $w_n=u_n$) 
		\[
		\int_{\mathbb{R}^N}
		u_n^2\log u_n^2\,dx
		=
		\int_{\mathbb{R}^N}
		(u_n-u)^2\log (u_n-u)^2\,dx
		+
		\int_{\mathbb{R}^N}
		u^2\log u^2\,dx
		+o(1).
		\]
		Using these identities, for sufficiently large $n$ the numerator of $Q(u_n)$ can be rewritten
		in the form
		\[
		\delta[u_n]=\delta[u_n-u]+\delta[u] + F\left(\|u_n-u\|_2^2, \|u\|_2^2\right)+o(1)\,,
		\]
		where
		\[
		F(x,y):=\pi\left((x+y)\log(x+y)-x\log x-y\log y\right)\qquad\forall x,y>0\,.
		\]
		By Remark \ref{rem:d_assunta}, Lemma \ref{lem:m=max} and \eqref{mn>m}, for large $n$ the denominator becomes
		\[
		d(u_n,\mathcal M)^2
		=
		\|u_n-u\|_2^2
		+
		\|u\|_2^2
		-
		m(u_n-u)
		+o(1)\,,
		\]
		so that
		\[
		Q(u_n)=\frac{\delta[u_n-u]+\delta[u]+F\left(\|u_n-u\|_2^2,\|u\|_2^2\right)+o(1)}{d(u_n-u,\mathcal M)^2+\|u\|_2^2+o(1)}\,.
		\]
		Now, from
		\[
		Q(u_n)=Q+o(1)
		\]
		and
		\[
		\frac{\delta[u_n-u]}{d(u_n-u,\mathcal M)^2}\ge Q,
		\]
		it follows
		\[
		\frac{\delta[u]+F\left(\|u_n-u\|_2^2,\|u\|_2^2\right)}{\|u\|_2^2}\le Q+o(1)\,.
		\]
		Let now $c_n>1$ be such that
		\[
		m(c_n u)=c_n^2m(u)=m(u_n-u),
		\]
		and set
		\[
		U_n:=c_n u.
		\]
		Since
		\[
		\|U_n\|_2^2=c_n^2\|u\|_2^2>\|u\|_2^2
		\]
		and, for every fixed $x>0$, the function
		\[
		f(y):=\frac{F(x,y)}{y}
		\]
		is monotonically decreasing, one has
		\[
		\frac{F\left(\|u_n-u\|_2^2,\|u\|_2^2\right)}{\|u\|_2^2}>\frac{F\left(\|u_n-u\|_2^2,\|U_n\|_2^2\right)}{\|U_n\|_2^2}.
		\]
		By \cite[Lemma 2.4]{Konig_JEMS}, we then obtain
		\[
		\begin{split}
			Q+o(1)=\frac{\delta[u_n]}{d(u_n,\mathcal M)^2}=&\,\frac{\delta[u_n-u]+\delta[u]+F\left(\|u_n-u\|_2^2,\|u\|_2^2\right)+o(1)}{d(u_n-u,\mathcal M)^2+\|u\|_2^2+o(1)}\\
			>&\,
			\frac{\delta[u_n-u]+\delta[U_n]+F\left(\|u_n-u\|_2^2,\|U_n\|_2^2\right)+o(1)}{d(u_n-u,\mathcal M)^2+\|U_n\|_2^2+o(1)}\,.
		\end{split}
		\]
		We observe that, since by construction $(c_n)_n$ is uniformly bounded, the functions
		\[
		\tilde u_n:=U_n+u_n-u
		\]
		are uniformly bounded in $H^1(\R^N)$. Therefore, the last term in the previous inequality is equal to $Q(\tilde u_n)+o(1)$. Indeed, denoting as usual $h(s)=s^2\log s^2$, the identity
		\[
		\int_{\R^N}h(u_n-u)\,dx+\int_{\R^N}h(U_n)\,dx=\int_{\R^N}h(\tilde u_n)\,dx+o(1)
		\] 
		can be obtained by Remark \ref{rem:BL} with $w_n=cu+(u_n-u)$, where $\displaystyle c=\lim_n c_n$ (up to subsequences), together with the identities $\displaystyle \int_{\R^N}h(cu)\,dx=\int_{\R^N}h(U_n)\,dx+o(1)$ (that follows by the convergence of $c_n$ to $c$) and $\displaystyle \limsup_n\left|\int_{\R^N}h(c_nu+(u_n-u))-h(cu+(u_n-u))\,dx\right|=0$ (that follows by \eqref{eq:condBL} and $\displaystyle \int_{\R^N}\Phi(|c_n-c|u)\,dx\to0$ as $n\to\infty$, with $\Phi$ as in Remark \ref{rem:BL}). 
		Hence,
		\[
		\lim_n Q(\tilde u_n)<Q\,,
		\]
		which is the contradiction we sought. 
		
		The case $\displaystyle \lim_n m(u_n-u)<m(u)$ is handled analogously. Note again that, even though in \eqref{mn>m} we started extracting a subsequence, the result of the lemma holds for the whole sequence $(u_n)_n$, since from any subsequence we can extract a further subsequence for which the previous argument applies.
	\end{proof}
	\begin{proof}[End of the proof of Theorem \ref{thm:main}]
	Assume by contradiction that \eqref{eq:assurdo} holds. Then $u\not\equiv0$ and Lemma \ref{lem:mn=m} imply
	\[
	m(u_n-u)=m(u)+o(1)\,.
	\]
	Consider first the case (up to subsequences)
	\[
	\lim_n \|u_n-u\|_2^2>\|u\|_2^2\,.
	\]
	Arguing as in the proof of Lemma \ref{lem:mn=m}, we obtain
	\[
	Q+o(1)
	=
	\frac{
		\delta[u_n-u]+\delta[u]
		+
		F\left(\|u_n-u\|_2^2,\|u\|_2^2\right)
		+o(1)
	}{
		d(u_m-u,\mathcal M)^2+\|u\|_2^2+o(1)
	}\,,
	\]
	and since
	\[
	\frac{\delta[u_n-u]}{d(u_n-u,\mathcal M)^2}\ge Q,
	\]
	it follows (recalling also that $\delta[u]\geq0$)
	\[
	Q+o(1)
	\ge
	\frac{F\left(\|u_n-u\|_2^2,\|u\|_2^2\right)}{\|u\|_2^2}>\frac{F\left(\|u_n-u\|_2^2,\|U_n\|_2^2\right)}{\|U_n\|_2^2},
	\]
	where
	\[
	U_n:=c_n u
	\]
	and $c_n>1$ is such that
	\[
	\|U_n\|_2^2=c_n^2\|u\|_2^2=\|u_n-u\|_2^2.
	\]
	However,
	\[
	\frac{F(x,x)}{x}=\pi\frac{2x\log(2x)-x\log x-x\log x}{x}=2\pi\log 2 \qquad\forall x>0,
	\]
	and we thus obtain
	\begin{equation}
		\label{eq:Q>2log2}
		Q\ge 2\pi\log 2.
	\end{equation}
	Observe that, switching the roles of $u_n-u$ and $u$, the analogous argument leads to \eqref{eq:Q>2log2} also if one originally has
	\[
	\lim_{n\to\infty}
	\|u_n-u\|_2^2
	<
	\|u\|_2^2\,,
	\]
	and the same occurs if $\displaystyle \lim_n\|u_n-u\|_2^2=\|u\|_2^2$, without multiplying by any constant $c_n$ and simply passing to the limit in $\displaystyle F\left(\|u_n-u\|_2^2, \|u\|_2^2\right)$.
	Since \eqref{eq:Q>2log2} contradicts the a priori estimate reported in Section \ref{sec:Q<2log2}, 
	we conclude that \eqref{eq:assurdo} cannot hold. Hence $u_n$ converges strongly to $u$ in $L^2(\mathbb{R}^N)$ and $Q$ is attained by Proposition \ref{prop:convL2}.
	\end{proof}
	
	\section{Exponential decay and optimizers in $H^1(\R^N,d\gamma)$: proof of Theorem \ref{thm:main2}}
	\label{sec:main2}
	In this section we show that if $u\in H^1(\R^N)\setminus\M$ is an optimizer of $Q(\cdot)$, then $u$ decays exponentially fast at infinity, namely there exists $C>0$ such that
	\begin{equation}
		\label{eq:expdec}
		|u(x)|\leq C e^{-|x|}\qquad\forall x\in \R^N\,.
	\end{equation}
	Clearly, \eqref{eq:expdec} implies
	\begin{equation}
		\label{eq:momfin}
	\int_{\R^N}|x|^2u(x)^2\,dx<\infty\,.
	\end{equation}
	Hence, setting $v=u/\gamma$, since direct computations yield
	\begin{equation*}
		\int_{\R^N}u^2\,dx=\int_{\R^N}v^2d\gamma\,,\qquad\int_{\R^N}|x|^2u^2\,dx=\int_{\R^N}|x|^2v^2\,d\gamma\,,
	\end{equation*}
	\begin{equation}
	\begin{split}
		\int_{\R^N}|\nabla u|^2\,dx&\,=\int_{\R^N}|\nabla v|^2\,d\gamma+\pi^2\int_{\R^N}|x|^2v^2\,d\gamma-2\pi\int_{\R^N}\gamma^2 v\nabla v\cdot x\,dx\\
		&\,=\int_{\R^N}|\nabla v|^2\,d\gamma-\pi^2\int_{\R^N}|x|^2v^2\,d\gamma+N\pi\int_{\R^N}v^2\,d\gamma\,,
	\end{split}
	\end{equation}
	(where we used the identity $\displaystyle \nabla u(x)=\gamma(x)\nabla v(x)+v(x)\nabla \gamma(x)=\gamma(x)\left(\nabla v(x)-\pi v(x)x\right)$ and the divergence thoerem) and
	\[
	\begin{split}
		\int_{\R^N}u^2\log\frac{u^2}{\|u\|_{L^2(\R^N)}^2}\,dx&\,=\int_{\R^N}v^2\log\frac{\gamma^2v^2}{\|v\|_{L^2(\R^N,d\gamma)}^2}\,d\gamma\\
		&\,=\int_{\R^N}v^2\log\frac{v^2}{\|v\|_{L^2(\R^N,d\gamma)}^2}\,d\gamma-\pi\int_{\R^N}|x|^2v^2\,d\gamma\,,
	\end{split}
	\]
	then \eqref{eq:momfin} ensures that $v\in H^1(\R^N,d\gamma)$ and $Q_\gamma(v)=Q(u)=Q$, i.e. $Q$ is attained also in $H^1(\R^N,d\gamma)\setminus\M_\gamma$.
	
	Therefore, proving Theorem \ref{thm:main2} reduces to establishing \eqref{eq:expdec}. To this end, we begin with the next preliminary result.
	\begin{lemma}
		\label{lem:EL}
		Let $u\in H^1(\R^N)\setminus\M$ be such that $\|u\|_2=1$ and $Q(u)=Q$. Then there exist $h\in H^1(\R^N)$ and $C,c>0$ such that 
		\begin{equation}
			\label{eq:dech}
		|h(x)|\leq Ce^{-c|x|^2}\qquad\forall x\in\R^N
		\end{equation}
		for which it holds
		\begin{equation}
		\label{eq:EL}
		-\Delta u=\left(\pi\log u^2+N\pi+Q\right)u-Qh\qquad\text{in }\mathcal D'(\R^N)\,.
		\end{equation}
	\end{lemma}
	\begin{proof}
	Since $Q(w)\geq Q(u)=Q$ for every $w\in H^1(\R^N)\setminus\M$, then $u$ is a minimizer for the functional $F(w):=\delta[w]-Qd(w,\M)^2$ on $H^1(\R^N)\setminus \M$. Hence, for every $\varphi\in C_c^\infty(\R^N)$ and $t>0$, by $F(u+t\varphi)\geq F(u)$ we obtain
	\begin{equation}
		\label{eq:diff1}
	\frac{\delta[u+t\varphi]-\delta[u]}{t}\geq Q\frac{d(u+t\varphi,\M)^2-d(u,\M)^2}{t}\,.
	\end{equation}
	Since $\varphi\in C_c^\infty(\R^N)$ and $u\log u^2\in L_{\text{\normalfont loc}}^1(\R^N)$, the entropy term is differentiable along the direction $\varphi$, so passing to the limit as $t\to0^+$ on the left hand side yields $\delta'[u]\varphi$, and the usual difference quotient argument gives
	\begin{equation}
		\label{eq:d'}
	\delta'[u]\varphi=2\int_{\R^N}\nabla u\cdot \nabla\varphi-\pi u\log u^2-N\pi u\,dx\,.
	\end{equation}
	Let us thus focus on the right hand side of  \eqref{eq:diff1}. Since $u\not\in\M$, by Remark \ref{rem:d_assunta} there exists $b\in\R^N$ such that 
	\[
	d(u,\M)^2=\|u\|_2^2-m(u)=1-\left(\int_{\R^N}u(x)e^{-\frac\pi2|x-b|^2}\,dx\right)^2\,,
	\]
	that is the distance between $u$  and $\M$ is the distance between $u$ and the function $I(b)\gamma(\cdot-b)$, with $\displaystyle I(b):=\int_{\R^N}u(x)e^{-\frac\pi2|x-b|^2}\,dx$.
	Let then $B(u)\subset \R^N$ be the set of all such $b\in\R^N$, and $\displaystyle G(u):=\left\{I(b)\gamma(\cdot-b)\,:\, b\in B(u)\right\}$. Since $m(u)>0$ and $\displaystyle e^{-\frac\pi2|\,\cdot\,-b|^2}\rightharpoonup0$ in $L^2(\R^N)$ as $|b|\to\infty$, the set $B(u)$ is compact, and so is $G(u)$ in $H^1(\R^N)$. We then show that
	\begin{equation}
		\label{eq:derd}
		\lim_{t\to0^+}\frac{d(u+t\varphi,\M)^2-d(u,\M)^2}{t}=2\inf_{g\in G(u)}\int(u-g)\varphi\,dx\,.
	\end{equation}
	Indeed, let $g\in G(u)$ be fixed. Then, for every $t>0$, 
	\[
	\begin{split}
	d(u+t\varphi,\M)^2\leq \|u+t\varphi-g\|_2^2&\,=\|u-g\|_2^2+2t\int_{\R^N}(u-g)\varphi\,dx+t^2\|\varphi\|_2^2\\
	&\,=d(u,\M)^2+2t\int_{\R^N}(u-g)\varphi\,dx+t^2\|\varphi\|_2^2\,,
	\end{split}
	\]
	so that 
	\[
	\limsup_{t\to0^+}\frac{d(u+t\varphi,\M)^2-d(u,\M)^2}{t}\leq2\int_{\R^N}(u-g)\varphi\,dx\,,
	\]
	and taking the infimum over $g\in G(u)$ 
	\begin{equation}
	\label{eq:limsup}
		\limsup_{t\to0^+}\frac{d(u+t\varphi,\M)^2-d(u,\M)^2}{t}\leq2\inf_{g\in G(u)}\int_{\R^N}(u-g)\varphi\,dx\,.
	\end{equation}
	Conversely, let $t_n\to0^+$ as $n\to\infty$ be such that 
	\[
	\lim_n\frac{d(u+t_n\varphi,\M)^2-d(u,\M)^2}{t_n}=\liminf_{t\to0^+}\frac{d(u+t\varphi,\M)^2-d(u,\M)^2}{t}\,.
	\]
	For every $n$, let $g_n\in G(u+t_n\varphi)$. Since $u+t_n\varphi$ is uniformly bounded in $L^2(\R^N)$, so is $g_n$, and recalling that $g_n(\cdot)=a_n\gamma(\cdot-b_n)$ for suitable $a_n\in\R$, $b_n\in\R^N$ shows that $a_n$ is uniformly bounded. Moreover, since $u+t_n\varphi\to u$ in $L^2(\R^N)$, it is easily seen that $m(u+t_n\varphi)\to m(u)$ as $n\to\infty$. Since $\displaystyle m(u+t_n\varphi)=\left(\int_{\R^N}(u+t_n\varphi)\gamma(x-b_n)\,dx\right)^2$ and $m(u)>0$, as $n\to\infty$ the sequence $(b_n)_n$ must be bounded in $\R^N$ (if not, we would have $g_n\rightharpoonup0$ in $L^2(\R^N)$, in turn yielding $m(u+t_n\varphi)\to0$). Hence, up to subsequences $a_n\to a_0$ and $b_n\to b_0$ as $n\to\infty$, so that $g_n\to g_0$ in $L^2(\R^N)$ with $g_0\in G(u)$. As a consequence, 
	\[
	\begin{split}
	d(u+t_n\varphi)^2=\|u+t_n\varphi-g_n\|_2^2&\,=\|u-g_n\|_2^2+2t_n\int_{\R^N}(u-g_n)\varphi\,dx+t_n^2\|\varphi\|_2^2\\
	&\,\geq d(u,\M)^2+2t_n\int_{\R^N}(u-g_n)\varphi\,dx+t_n^2\|\varphi\|_2^2
	\end{split}
	\]
	leading to 
	\[
	\begin{split}
\liminf_{t\to0^+}\frac{d(u+t\varphi,\M)^2-d(u,\M)^2}{t}=&\,\lim_n\frac{d(u+t_n\varphi,\M)^2-d(u,\M)^2}{t_n}\\
\geq&\, 2\int_{\R^N}(u-g_0)\varphi\,dx\geq 2\inf_{g\in G(u)}\int_{\R^N}(u-g)\varphi\,dx\,.
\end{split}
	\]
	Combining with \eqref{eq:limsup} gives \eqref{eq:derd}, that together with \eqref{eq:diff1} yields
	\[
	\delta'[u]\varphi\geq 2Q\inf_{g\in G(u)}\int_{\R^N}(u-g)\varphi\,dx\,.
	\]
	Applying the above inequality to $-\varphi$ and coupling with the previous one we obtain
	\begin{equation}
		\label{eq:boundd'}
		2Q\inf_{g\in G(u)}\int_{\R^N}(u-g)\varphi\,dx\leq \delta'[u]\varphi\leq 2Q\sup_{g\in G(u)}\int_{\R^N}(u-g)\varphi\,dx
	\end{equation}
	for every $\varphi\in C_c^\infty(\R^N)$. Since 
	\[
	\begin{split}
	&\inf_{g\in G(u)}\int_{\R^N}(u-g)\varphi\,dx=\int_{\R^N}u\varphi\,dx-\sup_{g\in G(u)}\int_{\R^N}g\varphi\,dx\\
	&\sup_{g\in G(u)}\int_{\R^N}(u-g)\varphi\,dx=\int_{\R^N}u\varphi\,dx-\inf_{g\in G(u)}\int_{\R^N}g\varphi\,dx\,,
	\end{split}
	\]
	\eqref{eq:boundd'} can be rewritten as 
	\[
	\inf_{g\in G(u)}\int_{\R^N}g\varphi\,dx\leq \Lambda(\varphi)\leq \sup_{g\in G(u)}\int_{\R^N}g\varphi\,dx
	\]
	with $\displaystyle \Lambda(\varphi):=\int_{\R^N}u\varphi\,dx-\frac{\delta'[u]\varphi}{2Q}$. Note that, since
	\[
	g\in G(u)\mapsto\int_{\R^N}g\varphi\,dx
	\]
	is a linear and continuous map on $L^2(\R^N)$, if we denote by $\mathcal{C}$ the closed convex envelope of $G(u)$ in $L^2(\R^N)$, the above chain of inequalities extends to
	\[
	\inf_{g\in\mathcal C}\int_{\R^N}g\varphi\,dx\leq \Lambda(\varphi)\leq \sup_{g\in \mathcal C}\int_{\R^N}g\varphi\,dx\,.
	\]
	Now, for every $g\in\mathcal{C}$ we set $\displaystyle L_g(\varphi):=\int_{\R^N}h\varphi\,dx$ and take $\mathcal{K}:=\{L_g\,:\,h\in\mathcal{C}\}\subset\mathcal{D}'(\R^N)$, so that the previous formula reads
	\begin{equation}
	\label{eq:L}
	\inf_{g\in\mathcal{C}}L_g(\varphi)\leq \Lambda(\varphi)\leq\sup_{g\in\mathcal{C}}L_g(\varphi).
	\end{equation}
Since $\mathcal{C}$ is closed and convex in $L^2(\R^N)$, so is $\mathcal{K}$ with respect to the weak$^*$-topology on $\mathcal{D}'(\R^N)$ by compactness of $G(u)$ and Krein’s theorem. Hence, $\Lambda\in\mathcal{K}$. Indeed, if it were $\Lambda\not\in\mathcal{K}$, then by the Hahn-Banach Separation Theorem applied in $\mathcal{D}'(\R^N)$ to the sets $\mathcal{K}$ and $\left\{\Lambda\right\}$ we would reach a contradiction with \eqref{eq:L} . Therefore, there exists $h\in\mathcal{C}$ such that
	\[
	\Lambda(\varphi)=L_h(\varphi)=\int_{\R^N}h\varphi\,dx\qquad\forall \varphi\in C_c^\infty(\R^N)\,.
	\]
	By the definition of $\Lambda$ and \eqref{eq:d'}, this proves \eqref{eq:EL}, and \eqref{eq:dech} follows by the fact that $h$ belongs to the convex envelope of a compact set of gaussian functions.
 	\end{proof}
 	
 	\begin{remark}
 		\label{rem:reg} Lemma \ref{lem:EL} ensures that any $u\in\ H^1(\R^N)$ with $\|u\|_2=1$ and $Q(u)=Q$ satisfies
 		\begin{equation}
 		\label{eq:lim=0}
 		\lim_{|x|\to\infty}u(x)=0\,.
 		\end{equation}
 		Indeed, set
 		\[
 		f(s):=\left(\pi\log s^2+N\pi+Q\right)s\,,
 		\]
 		so that \eqref{eq:EL} becomes
 		\[
 		-\Delta u = f(u)-Qh\qquad\text{in }\mathcal D'(\R^N)\,.
 		\]
 		We show first that $u\in L^\infty(\R^N)$. Note that for every $\theta\in(0,1)$ there exists $C_\theta>0$ such that $|f(s)|\leq C_\theta\left(|s|^{1-\theta}+|s|^{1+\theta}\right)$ for every $s\in\R$. Hence, if $N=1,2$, since $u\in H^1(\R^N)$ implies $u\in L^p(\R^N)$ for every $p\geq2$, the same is true for $f(u)-Qh$, in turn yielding $u\in W^{2,p}(\R^N)$ for every $p\geq2$, and thus $u\in L^\infty(\R^N)$ by Sobolev embeddings.  If $N\geq 3$, then $u\in L^{p}(\R^N)$ for every $p\in [2,2^*]$, and thus
 		\[
 		\sup_{y\in\R^N}\|u\|_{L^p(B_1(y))}<\infty\,.
 		\]
 		Let then $y\in\R^N$ and fix $p_0\in(2,2^*]$ such that
 		\begin{equation}
 		\label{eq:condp}
 		2<2p_0<N+2\,.
 		\end{equation}
 		By the above pointwise estimate on $f(s)$ with $\theta=2/N$ and the immersion of $L^{\frac{p_0}{1-2/N}}(B_1(y))$ in $L^{\frac{p_0}{1+2/N}}(B_1(y))$, we have $f(u)-Qh\in L^{\frac{p_0}{1+2/N}}(B_1(y))$, yielding $u\in W^{2,\frac{p_0}{1+2/N}}(B_1(y))$ and in turn $u\in L^{p_1}(B_1(y))$, with
 		\[
 		p_1=\frac{Np_0}{N+2-2p_0}\,,
 		\]
 		Observe that $p_1>p_0$ by \eqref{eq:condp}, and since $u\in L^{p_1}(B_1(y))$ the above argument gives immediately $\displaystyle u\in W^{2,\frac{p_1}{1+2/N}}(B_1(y))$. If $\displaystyle p_1/(1+2/N)>N/2$, the immersion of $\displaystyle W^{2,\frac{p_1}{1+2/N}}(B_1(y))$ in $L^\infty(B_1(y))$ and the fact that all these estimates do not depend on the center $y$ of the ball $B_1(y)$ entail $u\in L^\infty(\R^N)$. Conversely, if $\displaystyle p_1/(1+2/N)<N/2$, then \eqref{eq:condp} remains true with $p_1$ in place of $p_0$, and we can repeat the above bootstrap argument (note that one can always avoid the case $\displaystyle p_1/(1+2/N)=N/2$ by slightly modifying the original $p_0$ if necessary). In general, if after $k$ iterations we obtain an exponent $p_k$ such that $p_k/(1+2/N)<N/2$, then \eqref{eq:condp} holds with $p_k$ in place of $p_0$ and a further iteration can be done to obtain the new exponent
 		\[
 		p_{k+1}=\frac{N p_k}{N+2-2p_k}\,.
 		\] 
 		Since
 		\[
 		p_{k+1}-p_k=p_k\left(\frac {2p_k-2}{N+2-2p_k}\right)> \frac4{N-2}\qquad\forall k\,,
 		\]
 		after finitely many steps one arrives at $\displaystyle p_{k+1}/(1+2/N)>N/2$, which is enough to conclude $u\in L^\infty(\R^N)$.
 		Since $u\in L^\infty(\R^N)$, so is $f(u)$, and thus $u\in W^{2,q}(\R^N)$ for every $q\geq2$. Hence, $u\in C^{1,\alpha}(\R^N)$ for every $\alpha\in(0,1)$. In particular, $u$ is uniformly continuous on $\R^N$. This, together with the fact that $u\in L^2(\R^N)$, proves \eqref{eq:lim=0}. 
 		\end{remark}
		\begin{proof}[Proof of Theorem \ref{thm:main2}]
		Combining Kato's inequality
		\[
		-\Delta|u|\leq\text{\normalfont sgn}(u)(-\Delta u)
		\]
		with \eqref{eq:EL} yields
		\[
		-\Delta|u|\leq\left(\pi\log u^2+N\pi+Q\right)\text{\normalfont sgn}(u)u-\text{\normalfont sgn}(u)Qh\leq\left(\pi\log u^2+N\pi+Q\right)|u|+Q|h|\,. 
		\]
		Fix $\mu>1$ and note by Remark \ref{rem:reg} there exists $R>0$ such that 
		\[
		\pi\log u^2+N\pi+Q<-\mu\qquad\text{on }\R^N\setminus B_R\,,
		\]
		so that by \eqref{eq:dech}, in the sense of distributions, 
		\begin{equation}
		\label{eq:kato_u}
		-\Delta|u|+\mu|u|\leq Q|h|\leq QCe^{-c|x|^2}\qquad\forall x\in \R^N\setminus B_R\,.
		\end{equation}
		Consider then the function
		\[
		v(x):=Ae^{-(|x|-R)}\qquad\forall x\in \R^N\setminus B_R\,,
		\]
		where $A>0$ is chosen to ensure
		\begin{equation}
			\label{eq:v>|u|}
		v\geq|u|\qquad\text{on }\partial B_R
		\end{equation}
		(which is possible again by Remark \ref{rem:reg}). Since 
		\[
		\Delta v =\left(1-\frac{N-1}{|x|}\right)v\,,
		\]
		we obtain
		\begin{equation}
		\label{eq:diff_v}
		-\Delta v+\mu v = \left(\mu-1+\frac{N-1}{|x|}\right)v\geq(\mu-1)v\geq QCe^{-c|x|^2}\qquad\forall x\in\R^N\setminus B_R
		\end{equation}
		up to possibly further increasing $R$, since $e^{-c|x|^2}/v(x)$ tends to zero for every $c>0$ as $|x|\to\infty$.  
		
		Now, by Remark \ref{rem:reg}, for every $\varepsilon>0$ there exists $R_\varepsilon>0$ such that 
		\begin{equation}
			\label{eq:|u|<eps}
		|u(x)|\leq\varepsilon\qquad\forall x\in \R^N\setminus B_{R_\varepsilon}\,.
		\end{equation}
		Let then $\rho>R_\varepsilon$ and set $w_\varepsilon:=|u|-v-\varepsilon$. By \eqref{eq:kato_u} and \eqref{eq:diff_v}
		\[
		-\Delta w_\varepsilon+\mu w_\varepsilon\leq-\mu\varepsilon<0\qquad\text{on }B_\rho\setminus B_R\,,
		\]
		whereas by \eqref{eq:v>|u|} and \eqref{eq:|u|<eps}
		\[
		w_\varepsilon<0\qquad\text{on }\partial B_R\cup\partial B_\rho\,.
		\]
		Hence, the weak maximum principle yields $w_\varepsilon\leq0$ on $B_\rho\setminus B_R$, i.e.
		\[
		|u|\leq v+\varepsilon\qquad\text{on }B_\rho\setminus B_R
		\]
		(the estimate holds everywhere as both $u$ and $v$ are continuous functions). Since $\rho\to\infty$ as $\varepsilon\to0$, it then follows
		\[
		|u|\leq v\qquad\text{on }\R^N\setminus B_R\,,
		\] 
		that together with the continuity of $u$ implies \eqref{eq:expdec}, and we conclude.
		\end{proof}

\end{document}